\documentclass{article}

\usepackage[english]{babel}

\usepackage[a4paper,top=2cm,bottom=3cm,left=2.5cm,right=2.5cm]{geometry}

\usepackage{tikz}
\usepackage{adjustbox}
\usepackage{csquotes}
\usepackage{placeins,paralist}
\usepackage[inline]{enumitem}
\usepackage{amsthm,amsmath,amsfonts,amssymb}
\usepackage{mathtools, nccmath}
\usepackage{graphicx, booktabs, multirow}
\usepackage[ruled,vlined]{algorithm2e}
\usepackage{authblk}
\usepackage{caption}
\usepackage{subcaption}

\usetikzlibrary{calc,trees, mindmap, shapes, arrows}
\usetikzlibrary{arrows.meta}

\newtheorem{theorem}{Theorem}[section]
\newtheorem{lemma}[theorem]{Lemma}
\theoremstyle{definition}

\usepackage[colorlinks=true, allcolors=blue]{hyperref}
\usepackage[style=authoryear-icomp,maxbibnames=9,maxcitenames=2,uniquelist=false,backend=biber]{biblatex}
\newlist{choices}{enumerate*}{1}
\setlist[choices]{itemsep = 1.125in, label=(\Alph*)}

\providecommand{\topP}[1]{TOP-ST-MIN-P#1}
\providecommand{\topPL}[1]{TOP-ST-MIN-PL#1}

\title{The Team Orienteering Problem with Service Times and Mandatory \& Incompatible Nodes}

\author[1]{Alberto Guastalla}
\author[1]{Roberto Aringhieri}
\author[1]{Pierre Hosteins}
\affil[1]{Department of Computer Science, University of Turin \\
Corso Svizzera 185, 10149 Torino, Italy}
\date{\today}

\begin{document}

\maketitle

\begin{abstract}
The Team Orienteering Problem with Service Times and Mandatory \& Incompatible Nodes 
(TOP-ST-MIN) is a variant of the classic Team Orienteering Problem (TOP), which includes 
three novel features that stem from two real-world problems previously studied by the 
authors.
We prove that even finding a feasible solution is NP-complete.
Two versions of this variant are considered in our study.
For such versions, we proposed two alternative mathematical formulations, a mixed and a 
compact formulations.
Based on the compact formulation, we developed a Cutting-Plane Algorithm (CPA) exploiting 
five families of valid inequalities. 
Extensive computational experiments showed that the CPA outperforms CPLEX in solving the 
new benchmark instances, generated in such a way to evaluate the impact of the three 
novel features that characterise the problem.
The CPA is also competitive for the TOP since it is able to solve almost the same 
number of instances as the state-of-art algorithms.
\end{abstract}


\section{Introduction}
The Team Orienteering Problem (TOP) is a well-known combinatorial optimisation problem \parencite{Butt1994}, which belongs to the general class of Vehicle Routing Problems with Profits \parencite{OP2019}. The TOP that can be formulated by using a complete undirected graph in which each node has a profit and each arc has a travel time. The aim is to identify a given number of routes that maximise the profit from the nodes visited without exceeding a total time budget for all routes. Each route starts from the source node and ends to the destination one. 
Each node can be visited at most once except for the source and the destination. 
The TOP can be viewed as a basic framework that can be extended to address real-world problems 
that can be formulated as a variant of the TOP with some novel features. 


The Ambulances Routing Problem (ARP) consists in the management of the ambulance fleet for finding the tours to transport the injured to hospitals after a disaster \parencite{Aringhieri2021}.
Two groups of patients based on a triage system can be considered in the area affected by the disaster. \emph{Red} patients should be transported to hospitals because they suffer from serious injuries. On the contrary, \emph{green} patients are people who are slightly injured that can be cured in the field. Priority for green patients are introduced as profits to define who needs to be taken care of more urgently. Therefore, each green patient has a profit that can be interpreted as a sort of urgency level weight. Maximising the overall profit for green patients can be considered as a goal to provide an efficient service. On the contrary, red patients must be served thus they identify a set of mandatory nodes. 
The TOP formulation of the ARP introduces service times at nodes in order to consider the amount of time necessary to take care of the patients, a set of mandatory nodes to model red patients and a set of incompatibilities between nodes that force to reach a hospital after having visited one of them.
Another problem that can be formulated as a variant of the TOP is the Daily Swab Test Collection Problem (DSTCP) that arises in the daily collection of swab tests at home \parencite{Aringhieri2022}.

In this paper, we present a new variant of the TOP that generalises all the features introduced in the ARP and the DSTCP. Such a variant is called the Team Orienteering Problem with Service Times and Mandatory \& Incompatible Nodes (TOP-ST-MIN).

The TOP-ST-MIN considers (i) a \emph{service time} on each node, (ii) a set of \emph{mandatory nodes} and (iii) a set of \emph{physical incompatibilities} between nodes. In other words, it is necessary to explore each node for a specific fixed amount of time, some nodes must be compulsory visited, and a subset of arcs is forbidden in the graph. 
In accordance with the literature (see, e.g., \cite{Palomo2015}), we also consider \emph{logical incompatibilities} to provide a more comprehensive analysis.
Such incompatibilities state that nodes of different types cannot be served by the same route. For instance, in the healthcare setting, patients with different infectious diseases can not share the same vehicles. 
In the remainder of the paper we will consider two versions of the problem that could represent different operative contexts, that is the \topP{,} which considers only physical incompatibilities, and the \topPL{,} which includes also the logical ones.


The main contributions of this paper are:
(i) the introduction of a new variant of the TOP called TOP-ST-MIN derived from real-world problems, 
(ii) the proof of NP-completeness of the feasibility problem,
(iii) two new alternative mathematical formulations for the TOP-ST-MIN,
(iv) an exact algorithm based on a cutting-plane approach exploiting also new valid inequalities,
(v) several sets of benchmark instances generated in such a way to evaluate the impact of the new three features, and
(vi) an extensive computational analysis including a comparison with the state of the art algorithms for the TOP.

The paper is organised as follows. 
Section~\ref{sec:review} reports a comprehensive literature analysis and some complexity 
results to demonstrate the novelty and the qualitatively stronger difficulty of the proposed 
variant.
Section~\ref{sec:models} presents two alternative mathematical formulations for TOP-ST-MIN, the second of which is a compact formulation better suited to the exact approach proposed.
Section~\ref{sec:exact-method} proposes an exact algorithm for solving the TOP-ST-MIN based on the cutting-plane methodology that adopts several families of new and standard valid inequalities to strengthen the linear programming formulation. We also introduce a new separation method for the subtour elimination constraints based on the elementary cycles detection on a directed graph.
Section~\ref{sec:instances} described several sets of new benchmark instances generated in such a way to evaluate the impact of the new three features on the optimal solution.
Based on these new benchmark instances, Section~\ref{sec:analysis} presents an extensive computational analysis to validate and to evaluate the proposed computational approach reporting
(i) a comparison between the two alternative mathematical formulations,
(ii) a full evaluation of the impact of the valid inequalities introduced,
(iii) an evaluation of the performance of the proposed exact approach, and
(iv) a comparison of an adapted version of our exact approach with the state-of-art algorithms for the TOP.
Finally, Section~\ref{sec:end} closes the paper.



\section{Literature review and complexity results}
\label{sec:review}

During the last decades, the literature regarding the orienteering problems has grown rapidly due to the fact that they are attractive to the research community. The purpose of this literature review is to focus our attention on the three additional features characterising the TOP-ST-MIN. Table \ref{tab:contributions} shows the most recent works we found and their corresponding general features (denoted by $1$ to $13$) and those related to the TOP-ST-MIN (denoted by A to E). All articles collected were published between the years 2015 and 2023 and approximately half of them are about the Orienteering Problem (OP), while the others are related to the TOP. 

\begin{table}[!ht]
\centering
\begin{adjustbox}{width=1\textwidth,center}
\begin{tabular}{@{}lccccccccccccccccccc@{}}
\toprule
& \multicolumn{13}{c}{\textbf{General features}} & \multicolumn{5}{c}{\textbf{TOP-ST-MIN features}} \vspace*{3pt} \\
\textbf{Reference} & \textbf{(1)} & \textbf{(2)} & \textbf{(3)} & \textbf{(4)} & \textbf{(5)} & \textbf{(6)} & \textbf{(7)} & \textbf{(8)} & \textbf{(9)} & \textbf{(10)} & \textbf{(11)} & \textbf{(12)} & \textbf{(13)} & \textbf{(A)} & \textbf{(B)} & \textbf{(C)} & \textbf{(D)} & \textbf{(E)} \\ \cmidrule(r){1-1} \cmidrule(lr){2-14} \cmidrule(l){15-19}
\cite{Palomo2015} & \checkmark & & & & &  & & & & & & & & & & \checkmark & & \checkmark \\ 
\cite{Kotiloglu2017} & & \checkmark & & &  & & & & & \checkmark & & & & \checkmark & & \checkmark & & \\
\cite{Lin2017}& & \checkmark & & &  & & & & & \checkmark & & & & \checkmark & & \checkmark & & \\
\cite{Palomo2017} & \checkmark & & & & &  & & & & & & & & & & \checkmark & & \checkmark \\
\cite{Lu2018} & \checkmark & & & & &  & & & & & & & & & & \checkmark & & \checkmark \\
\cite{Reyes2018} & & \checkmark & & & & & & & & & & & \checkmark & \checkmark & & & & \\
\cite{Liao2018} & \checkmark & & \checkmark & & & & & \checkmark & & \checkmark & & & \checkmark & \checkmark & & & & \\
\cite{Taylor2018} & \checkmark & & & & & & & & & & & & & \checkmark & & \checkmark & & \\
\cite{Stavropoulou2019} & & \checkmark & & &  & & & & & \checkmark & & \checkmark & & \checkmark & & \checkmark & & \\
\cite{Jin2019} & & \checkmark & & & & & & & & & & & \checkmark & \checkmark & & & & \\
\cite{Thompson2019} & & \checkmark & \checkmark &  &  \checkmark & & & & & & & & & \checkmark & & & & \\
\cite{Assunção2020} & & \checkmark & & & & & & & & & & & & & & \checkmark & & \\
\cite{Zheng2020} & \checkmark & & & & & & & & & \checkmark & & & & \checkmark & & & & \\
\cite{Wisittipanich2020} & & \checkmark & &  & & & & & & \checkmark & \checkmark & & & \checkmark & & & & \\
\cite{Gndling2020TimeDependentTT} & & \checkmark & \checkmark &  & \checkmark & & & & & \checkmark & & & & & \checkmark & & & \\
\cite{Wallace2020} & \checkmark & & &  & & & & \checkmark & & & & \checkmark & & & \checkmark & & & \\
\cite{Sebastia2020} & \checkmark & & &  & & & & \checkmark & & \checkmark & & & & & \checkmark & & & \\
\cite{Hanafi2020} & & \checkmark & & & & & \checkmark & \checkmark & \checkmark & & & & & \checkmark & & & & \\
\cite{Dutta2020} & \checkmark & & & \checkmark & & & & & & & \checkmark & & & \checkmark & & & & \\
\cite{Ruiz2021} & & \checkmark & &  & & \checkmark & \checkmark & & & \checkmark & & & & \checkmark & & & & \\
\cite{Ruiz2021_} & & \checkmark &  &  & & \checkmark & \checkmark & & & \checkmark & & & & \checkmark & & & & \\ 
\cite{Feo2021} & & \checkmark &  &  & & & \checkmark & & & & & & & \checkmark & & & & \\
\cite{Heris2021} & & \checkmark & & & \checkmark & & & & & \checkmark & \checkmark & & & \checkmark & & & & \\
\cite{Aringhieri2021} & & \checkmark &  &  &  &  & & & & & & & & \checkmark & & \checkmark & \checkmark & \\
\cite{Yu2021} & & \checkmark &  &  & \checkmark & & & & & & & & & \checkmark & & & & \\ 
\cite{Vu2021} & \checkmark & & &  &  & & & & \checkmark & \checkmark & & & & \checkmark & & \checkmark & & \checkmark \\
\cite{Porras2022} & \checkmark & & \checkmark &  & & & & & & \checkmark & & & & \checkmark & & & & \\
\cite{Aringhieri2022} & & \checkmark & &  & & & & & & & & & & \checkmark & & & & \\
\cite{Morandi2023} & \checkmark & & &  & & & & & & & & & & & \checkmark & & & \\
\cite{Corrêa2024} & & \checkmark &  &  &  & \checkmark & & & & \checkmark & & & & \checkmark & & \checkmark & & \\
\cmidrule(r){1-1} \cmidrule(lr){2-14} \cmidrule(l){15-19}
\textbf{This paper}  & & \checkmark &  &  &  &  & & & & & & & & \checkmark & & \checkmark & \checkmark & \checkmark \\
\bottomrule
\end{tabular}
\end{adjustbox}
\caption{The most recent contributions collected. The features considered are: (1) Single-Route, (2) Multi-Route, (3) Time-Dependent, (4) Clusters, (5) Variable-Profits, (6) Multi-Modal, (7) Multi-Agent, (8) Multi-Visit, (9) Precedence, (10) Time-Windows, (11) Multi-Objective, (12) Capacity, (13) Stochastic, (A) Fixed Service Times, (B) Variable Service Times, (C) Mandatory Nodes, (D) Physical Incompatibilities, (E) Logical Incompatibilities.}
\label{tab:contributions}
\end{table}

With respect to the TOP-ST-MIN, the most frequent feature found is the service times (fixed and variable) at each node, followed by the mandatory nodes and the incompatibilities. Most of the considered works are about the classic Tourist Trip Design Problem (TTDP) in which a tourist has to visit a subset of points of interest (POIs). Each POI is characterised by a profit, by a service time and can only be visited during a specific time window. Regarding the solution methods, we mainly found metaheuristic approaches, fewer exact algorithms as well as a few matheuristic and tailored heuristic methods.

Table~\ref{tab:contributions} clearly shows that TOP-ST-MIN has not been previously considered in the literature. The sole work that groups the all features introduced (except the physical incompatibilities) is \cite{Vu2021}, which addresses a variant of the TTDP that considers additional constraints such as logical incompatibilities and mandatory nodes. Finally, we would remark that all papers that consider variable service times are variants of the OP except for \cite{Gndling2020TimeDependentTT} and all the articles considering logical incompatibilities consist exclusively in variants of the OP.

The remaining of the section reports the literature review and some complexity results for the TOP-ST-MIN.
In Section~\ref{ssec:3features}, we reports and discusses those papers related to the TOP with at least one of the three 
additional features introduced in the TOP-ST-MIN, some of them are already included in Table~\ref{tab:contributions}.
In Section~\ref{ssec:complexity}, we prove that (i) the TOP-ST-MIN is NP-hard, and that (ii) the introduction of mandatory nodes and physical incompatibilities make the problem of finding a feasible solution NP-complete. The latter result further justifies the need to study the TOP-ST-MIN.

\subsection{Literature review}
\label{ssec:3features}

\paragraph{The service time at nodes.}

Originally, the basic OP and TOP framework does not consider any service time at each node. However, many problems in healthcare, tourism, and in general everything related to a routing problem, may need to consider a certain amount of waiting time at each node to model a specific service. 

The first related work involving this feature has been proposed by \cite{Tang2005}. The authors addressed the Selective Travelling Salesperson Problem with Stochastic service times and travel times (SSTSP). In the SSTSP, service and travel times follow two different continuous probability distributions. A non-negative reward for providing a service is associated with each customer and there is a specified limit on the duration of the solution route. The objective of the SSTSP is to design a priori route that visits a subset of the customers in order to maximize the total profit collected and the probability that the total actual route duration exceeds a given threshold is no larger than a chosen probability value. They formulated the SSTSP as a chance-constrained stochastic program and proposed both exact and heuristic approaches for solving it. The SSTSP is very similar to the Orienteering Problem with Stochastic Travel and Service times (OPSTS) described by \cite{Campbell2011} except for the fact that the source and destination nodes need to be the same.

\cite{EGL2013} introduced the Orienteering Problem with Variable Profits (OPVP) in which a single vehicle can collect the whole profit at the customer after a discrete number of ``passes'' or spending a continuous amount of time. As in the classical orienteering problem, the objective is to determine a maximal profit tour for the vehicle, starting and ending at the depot, and not exceeding a travel time limit.

The TTDP has been studied intensively becoming the most famous application of the OP. In \cite{Lim2018}, the authors proposed an algorithm, called \texttt{PersTour}, for recommending personalised routes using POIs popularity and user interest preferences, which are automatically derived from real-life travel sequences based on geo-tagged photographs. The underlying problem is modelled as an OP with time windows and variable service times. The user interest is based on the visit duration and the POI visit time can be personalised using this user interest measure.

The underlying idea of the TTDP has been extended in order to consider more routes subdivided into multiple days. In fact, \cite{Zheng2020} aimed to design personalised itineraries with hotel selection for multi-day urban tourists. They proposed a linear model to formulate the problem as a variant of the TOP basic framework in which the objective consists in finding the optimal set of daily trips, which start from an hotel and end to another one connecting multiple day paths into a single big path. 

Focusing on the service times at nodes, tourism is undoubtedly one of the most relevant fields in the OP literature because of the investigation of many variants and applications of the TTDP framework. However, there are other research areas, such as healthcare \parencite{Aringhieri2022b}, where the OP or TOP framework with service times could be used. Further, \cite{Jin2019} addressed the intrahospital routing of the phlebotomists. The phlebotomists are responsible for drawing specimens from patients based on doctors orders. The authors formulated the phlebotomist intrahospital routing problem as a team orienteering problem with stochastic rewards and service times. They present an a priori solution approach and derive a method for efficiently sampling the value of a solution.

\paragraph{The mandatory nodes.}
In a real application context, it could be necessary to visit some critical or relevant subset of nodes. In a tourism context, all tourists may share some common interests, which the route planner may impose as mandatory POIs. During an emergency situation, the ambulances have to collect all the patients that suffer of severe pathologies or injuries from a serious accident. For this reason, in a mathematical model, these patients could be modelled as mandatory nodes in a graph.

The first work we found concerning the mandatory nodes has been proposed by \cite{Kotiloglu2017}. The authors described the complex selection for generating personalised route recommendations for tourists. The aim of their work is to generate routes that contain all mandatory POIs maximising the total profit collected from the optional points visited daily. The problem takes also into account different day availabilities, opening hours, limitations on the tour lengths, budgets and other restrictions.

\cite{Taylor2018} formalised the idea of including this notion of mandatory nodes stating a variant of the TTDP with mandatory POIs: the Tour Must See Problem (TMSP). The TMSP is basically a variant of the OP in which there is a subset of mandatory nodes to visit. They proposed an integer linear model for solving the TMSP. To formalise the idea of mandatory nodes, \cite{Assunção2020} proposed a new variant of the TOP called the Steiner Team Orienteering Problem (STOP). The aim of the problem remains the same as for the TOP with the difference that there is a specified subset of mandatory nodes that must be visited. The authors proposed a novel commodity-based formulation for the problem and solve it through a cutting-plane scheme.  

Considering the applications in the transportation area, \cite{Stavropoulou2019} formulated a new mathematical model called the Consistent Vehicle Routing Problem with Profits (CVRPP). The problem involves two sets of customers, the frequent customers that are considered mandatory and the non-frequent and potential ones with known and estimated profits, respectively. Both have known demands and service requirements over a planning horizon of multiple days. The objective is to determine the vehicle routes that maximise the profit while satisfying vehicle capacity, route duration and consistency constraints. To tackle with it, the authors developed an adaptive tabu-search using both short and long-term memory structures to guide the search process.

A new variant of the TOP has been proposed by \cite{Lin2017}. The authors investigated the TOP with time windows and mandatory visits (TOPTW-MV). In the TOPTW-MV, some customers are important customers that must be visited. The others are called optional customers. Each customer carries a positive profit. The goal is to determine a given number of paths that maximise the total profit collected at visited nodes, while observing some constraints such as mandatory visits and time window intervals. They formulated a mathematical linear model for TOPTW-MV and designed a multi-start simulated annealing metaheuristic to solve the problem.

\paragraph{The incompatibilities between nodes.}

To the best of our knowledge, the physical incompatibilities are never considered in literature before \cite{Aringhieri2021}. As reported in their work, it is possible to recognise them in the situation in which the ambulances must transport the most injured patients on the hospitals. For this reason, some connections are forbidden and therefore the graph is not complete. 
On the other hand, the logical incompatibilities can be used to separate several types of patients. For example, patients with different infectious diseases can not share the same ambulance.

The first work we found that considers the logical incompatibilities is \cite{Palomo2015}. The authors addressed a variant of the OP with mandatory visits and incompatibilities among nodes. They proposed a hybrid algorithm based on a reactive GRASP and a general VNS. Finally, they also validate the performance of the proposed algorithm on some instances taken from the literature of the traditional OP. 

\cite{Lu2018} proposed a further variant of the OP called the OP with Mandatory Visits and Exclusionary Constraints (OPMVEC). The aim of this problem is to visit a set of mandatory nodes and some optional nodes, while respecting the logical incompatibilities between nodes and the maximum total time budget constraint. The authors presented a highly effective memetic algorithm for the OPMVEC combining (i) a dedicated tabu-search procedure that considers both feasible and infeasible solutions by constraint relaxation, (ii) a backbone-based crossover, and (iii) a randomised mutation procedure to prevent from premature convergence. 

Another work that includes the notion of logical incompatibilities is \cite{Vu2021}. The authors enriched the TTDP with further constraints: (i) the addition of some mandatory locations, (ii) the categorisation of each POI, (iii) the restriction on the number of locations of each category that can be visited, and (iv) the order in which selected locations are visited which introduces the incompatibilities. They proposed a branch-and-check approach in which the reduced problem selects a subset of POIs, verifying all but time-related constraints. These locations define candidate solutions for the problem. For each of them, the sub-problem checks whether a feasible trip can be built using the given locations. 

\subsection{Some complexity results}
\label{ssec:complexity}

The proof that TOP-ST-MIN is NP-hard is quite straightforward and obtained by reducing the TOP to the TOP-ST-MIN. Every instance of the TOP can be mapped into an instance of the TOP-ST-MIN simply considering an empty set of mandatory nodes and incompatibilities and setting to zero all the service times. It follows that the TOP-ST-MIN results to be at least hard as the TOP. Finally, knowing that the TOP has been proven to be NP-hard by \cite{Butt1994}, the TOP-ST-MIN is also NP-hard.


In addition, we aim to demonstrate that the new features introduced in the TOP-ST-MIN make NP-complete the problem of finding a feasible solution in some instances, even for the single-path problem (OP-ST-MIN). 
The Feasibility Problem (FP) of the OP-ST-MIN consists in finding a single route from the source to the destination node which includes all mandatory nodes, does not violate any incompatibilities, and with a total route time less than or equal to a maximum threshold $T_{\max}$. We state the following:
\begin{theorem}
\label{ca}
Finding a feasible solution for the OP-ST-MIN is NP-complete.
\end{theorem}
\begin{proof}
In order to prove the above theorem, we provide a reduction to the well-known NP-complete Hamiltonian Path Problem (HPP)~\parencite{Karp1972}. An HPP instance consists of an undirected graph $G=(V,E)$ to which the following question is associated: is there a path in $G$ which contains all the nodes of $V$? Starting from a generic instance $\psi$ of HPP, we construct an instance $\phi$ of OP-ST-MIN as follows:
\begin{itemize}
    \item introduce a graph $G'(V',E')=G$ and add two nodes $s$ (source) and $t$ (destination) to $V'$  with an edge between $s$ (respectively $t$) and any node of $V$;
    \item to any node of $V'$, associate a service time of 0 and to each edge of $E'$ associate a travel time of 0. The maximum time available is set at $T_{\max}=0$;
    \item consider all nodes in $V'$ as mandatory nodes, so that the tour of the single team must contain all nodes;
    \item for each edge not present in $G$, we suppose a physical incompatibility between the two nodes. No logical incompatibility is present.
\end{itemize}
Suppose first that a Hamiltonian path $p$ exists in $G$, i.e., $\psi$ is a \emph{Yes} instance. We can build an OP-ST-MIN tour by connecting $s$ to the first node in $p$ and the last node in $p$ to $t$. $\phi$ is therefore a \emph{Yes} instance.

Conversely, suppose that $\phi$ is a \emph{Yes} instance, i.e., there exists a tour from $s$ to $t$. Removing nodes $s$ and $t$ from the tour, we obtain a path of $G$ containing all the nodes of $V$, which is therefore a Hamiltonian path for $G$ and $\psi$ is a \emph{Yes} instance, which completes the proof.
\end{proof}
We observe that the presence of mandatory nodes and physical incompatibilities makes the FP NP-complete. 
On the contrary, in a generic instance of the TOP (i.e., the graph is complete and no mandatory nodes 
are present), it is always possible to perform a tour directly from $s$ to $t$ without performing any 
additional stop. This proves that the TOP-ST-MIN is harder to tackle with respect to the TOP.


\section{The mathematical formulations}
\label{sec:models}

We present two different mathematical formulations for the TOP-ST-MIN, namely the mixed 
and the compact formulation. After introducing some notation in Section~\ref{ssec:notation}, we describe the
two alternative formulations proposed for the TOP-ST-MIN in Section~\ref{ssec:MTOP-ST-MIN} and~\ref{ssec:CTOP-ST-MIN}
showing first a model for the TOP-ST-MIN which considers only the physical incompatibilities 
(denoted by \topP{}) and then its
extension to include also the logical ones (denoted by \topPL{}) as stated in the introduction.

\subsection{Notation}
\label{ssec:notation}

We model the TOP-ST-MIN on a directed graph $G=(N,A)$ where $N$ is the set of nodes denoted from $1$ (the 
source or depot) to $n$ (the destination), and $A$ is the set of directed arcs across the nodes in $N$.
To ease the readability of the mathematical formulations, we introduce the set $\hat{N} = \{2, \ldots, n - 1\}$ to represent the set of \emph{customers} and the set $\hat{A} = \{(i,j) \in A : i \in \hat{N} \cup \{1\}, \: j \in \hat{N} \cup \{n\}, \: i \neq j \}$ to represent the set of \emph{traversable} arcs.
To each node $k \in \hat{N}$ is associated a non-negative amount of service time $s_k$ and a profit $p_k$.
Each arc has a non-negative travel time $t_{ij}$ across the node $i$ and $j$.
Further, we introduce the following three sets: let $M \subset \hat{N}$, $I \subset \hat{A}$, 
$C \subset \hat{N} \times \hat{N}$ be the set of 
mandatory nodes, the set of physical incompatibilities, and the set of logical incompatibilities, respectively.

The aim of the TOP-ST-MIN consists in finding a set of $m$ routes from the source node to the destination one maximising the overall 
profit collected in such a way that the total time (service and travel) of each route does not exceed the maximum allowed 
time $T_{\max}$, each mandatory node should be visited, and each route should not violate the incompatibilities considered.


\subsection{The mixed formulation}
\label{ssec:MTOP-ST-MIN}

Combining the vehicle-based formulation for the TOP \parencite{Vansteenwegen2011} and the flow-based formulation
for the TOP \parencite{bianchessi2018}, the mixed formulation for the \topP{} is the following:

\allowdisplaybreaks{
\begin{subequations}
\begin{align}
\max \quad & {\sum_{r=1}^m \sum_{k \in \hat{N}} p_k \: y_{kr}} \label{eq:MTOP-ST-MINof}\\
\textrm{s.t.}
\quad & \sum_{r=1}^{m} \sum_{(1,j) \in \hat{A}} x_{1jr} = \sum_{r=1}^{m} \sum_{(i,n) \in \hat{A}} x_{inr} = m, \label{eq:MTOP-ST-MINstartend}\\
\quad & \sum_{(i,k) \in \hat{A}} x_{ikr} = \sum_{(k,j) \in \hat{A}} x_{kjr} = y_{kr}, \quad \forall \: k \in \hat{N}, \quad r = 1,\ldots, m, \label{eq:MTOP-ST-MINconnectivity}\\
\quad & \sum_{r=1}^{m} y_{kr} \leq 1, \quad \forall \: k \in \hat{N}, \label{eq:MTOP-ST-MINonetour}\\
\quad & \sum_{r=1}^{m} y_{kr} = 1, \quad \forall \: k \in M, \label{eq:MTOP-ST-MINmandatory}\\
\quad & \sum_{r=1}^{m} x_{ijr} = 0, \quad \forall \: (i,j) \in I, \label{eq:MTOP-ST-MINIncompatibilities}\\
\quad & \sum_{(k,j) \in \hat{A}} z_{kjr} - \sum_{(i,k) \in \hat{A}} z_{ikr} = \sum_{(k,j) \in \hat{A}} (t_{kj} + s_k) \: x_{kjr}, \quad \forall \: k \in \hat{N}, \quad r = 1,\ldots, m, \label{eq:MTOP-ST-MINFlow}\\ 
\quad & z_{ijr} \leq (T_{\max} - s_j - t_{jn}) \: x_{ijr}, \quad \forall \: (i,j) \in \hat{A}, \quad r = 1,\ldots, m, \label{eq:MTOP-ST-MINUB}\\
\quad & z_{ijr} \geq (t_{1i} + s_i + t_{ij}) \: x_{ijr}, \quad \forall \: (i,j) \in \hat{A}, \quad r = 1,\ldots, m, \label{eq:MTOP-ST-MINLB}\\
\quad & z_{1kr} = t_{1k} \: x_{1kr} , \quad \forall \: k \in \hat{N}, \quad r = 1,\ldots, m, \label{eq:MTOP-ST-MINDepot}\\
\quad & 0 \leq x_{1nr} \leq m, \quad \forall \: r = 1,\ldots, m,\label{eq:MTOP-ST-MINDEF1}\\
\quad & 0 \leq x_{ijr} \leq 1, \quad \forall \: (i,j) \in \hat{A}, \quad r = 1,\ldots, m, \label{eq:MTOP-ST-MINDEF2}\\
\quad & x_{ijr} \in \mathbb{N^+}, \quad \forall \: (i,j) \in \hat{A}, \quad r = 1,\ldots, m,\label{eq:MTOP-ST-MINDEF3}\\
\quad & z_{ijr} \in \mathbb{R^+}, \quad \forall \: (i,j) \in \hat{A}, \quad r = 1,\ldots, m,\label{eq:MTOP-ST-MINDEF4}\\
\quad & y_{kr} \in \{0, 1\}, \quad \forall \: k \in \hat{N}, \quad r = 1,\ldots, m \label{eq:MTOP-ST-MINDEF5}
\end{align}
\label{mod:MTOP-ST-MIN}
\end{subequations}
}

The proposed model makes use of the following decision variables:
the variable $x_{ijr}$ is equal to $1$ if and only if the arc $(i,j)$ is traversed by the route $r$, $0$ otherwise;
the variable $y_{kr}$ is equal to $1$ if and only if the node $k$ is visited by the route $r$, $0$ otherwise;
the variable $z_{ijr}$ defines the arrival time at node $j$ coming from the node $i$ and can be thought as the 
amount of flow that passes through the arc $(i,j)$ inside the route $r$.

The objective function \eqref{eq:MTOP-ST-MINof} maximises the overall profit collected by each route.
Constraint~\eqref{eq:MTOP-ST-MINstartend} guarantees (with the optimisation sense of the 
objective function) that each route starts from the source node and ends to the destination one.
Constraints~\eqref{eq:MTOP-ST-MINconnectivity} guarantee the connectivity of each route while
the constraints~\eqref{eq:MTOP-ST-MINonetour} guarantee that a node can only be visited by at most one route.
Constraints~\eqref{eq:MTOP-ST-MINmandatory} ensure that all the mandatory customers $M$ will be visited.
Constraints~\eqref{eq:MTOP-ST-MINIncompatibilities} guarantee that all the physical incompatibilities 
are satisfied. 
Constraints~\eqref{eq:MTOP-ST-MINFlow} and~\eqref{eq:MTOP-ST-MINUB} are the classical Gavish-Graves (GG) 
subtours elimination constraints adapted to the TOP \parencite{Gavish2004}. 
Constraints~\eqref{eq:MTOP-ST-MINUB} and~\eqref{eq:MTOP-ST-MINLB} set the upper (lower) bound on the duration of the route $r$ while the 
constraints~\eqref{eq:MTOP-ST-MINDepot} bound the flow originating from the initial depot by the same route.
Constraints~\eqref{eq:MTOP-ST-MINDEF1}, \eqref{eq:MTOP-ST-MINDEF2}, \eqref{eq:MTOP-ST-MINDEF3}, 
\eqref{eq:MTOP-ST-MINDEF4} and \eqref{eq:MTOP-ST-MINDEF5} are the variable definition constraints.

In order to include the logical incompatibilities to model the \topPL{}, we can
add the constraints 
\begin{equation}
     |C_k| (1 - y_{kr}) \geq \sum_{i \in C_k} y_{ir}, \quad \forall \: k \in \hat{N}, 
     \quad \forall \: r = 1,\ldots, m \: ,
     \label{mod:MTOP-ST-MIN-logic-incomp}
\end{equation}
in the formulation~\eqref{mod:MTOP-ST-MIN}
where the set $C_k$ represents the subset of nodes that are logical incompatible with $k$. 
The constraints~\eqref{mod:MTOP-ST-MIN-logic-incomp} impose that only one set of nodes between $\{k\}$ 
and $C_k$ can be visited by the route $r$.

\subsection{The compact formulation}
\label{ssec:CTOP-ST-MIN}

Starting from the TOP flow-based formulation  \parencite{bianchessi2018}, the compact formulation for the \topP{} is the following:

\allowdisplaybreaks{
\begin{subequations}
\begin{align}
\max \quad & {\sum_{k \in \hat{N}} p_k \: y_k}\label{eq:CTOP-ST-MINof}\\
\textrm{s.t.}
\quad & \sum_{(i,j) \in \hat{A}} x_{1j} = \sum_{(i,n) 
\in \hat{A}} x_{in} = m, \label{eq:CTOP-ST-MINstartend}\\
\quad & \sum_{(i,k) \in \hat{A}} x_{ik} = \sum_{(k,j) 
\in \hat{A}} x_{kj} = y_{k}, \quad \forall \: k \in \hat{N}, \label{eq:CTOP-ST-MINconnectivity}\\
\quad & \sum_{(k,j) \in \hat{A}} z_{kj} - \sum_{(i.k) \in \hat{A}} z_{ik} = \sum_{(k,j) \in \hat{A}} (t_{kj} + s_k) \: x_{kj}, \label{eq:CTOP-ST-MINFlow} \quad \forall \: k \in \hat{N},\\ 
\quad & z_{ij} \leq (T_{\max} - s_j - t_{jN}) \: x_{ij}, \quad \forall \: (i,j) \in \hat{A}, \label{eq:CTOP-ST-MINUB}\\
\quad & z_{ij} \geq (t_{1i} + s_i + t_{ij}) \: x_{ij}, \quad \forall \: (i,j) \in \hat{A}, \label{eq:CTOP-ST-MINLB}\\
\quad & z_{1k} = t_{1k} \: x_{1k}, \quad \forall \: k \in \hat{N}, \label{eq:CTOP-ST-MINDepot}\\
\quad & y_{k} = 1, \quad \forall \: k \in M, \label{eq:CTOP-ST-MINmandatory}\\
\quad & x_{ij} = 0, \quad \forall \: (i,j) \in I, \label{eq:CTOP-ST-MINIncompatibilities}\\
\quad & 0 \leq x_{1N} \leq m,\label{eq:CTOP-ST-MINDEF1}\\
\quad & 0 \leq x_{ij} \leq 1, \quad \forall \: (i,j) \in \hat{A}, \label{eq:CTOP-ST-MINDEF2}\\
\quad & x_{ij} \in \mathbb{N^+}, \quad \forall \: (i,j) \in \hat{A},\label{eq:CTOP-ST-MINDEF3}\\
\quad & z_{ij} \in \mathbb{R^+}, \quad \forall \: (i,j) \in \hat{A},\label{eq:CTOP-ST-MINDEF4}\\
\quad & y_{k} \in \{0, 1\}, \quad \forall \: k \in \hat{N}\label{eq:CTOP-ST-MINDEF5}
\end{align}
\label{mod:CTOP-ST-MIN}
\end{subequations}
}

The proposed model makes use of the following decision variables:
the variable $x_{ij}$ is equal to $1$ if and only if the arc $(i,j)$ is traversed, $0$ otherwise;
the variable $y_{k}$ is equal to $1$ if and only if the node $k$ is visited, $0$ otherwise;
the variable $z_{ij}$ defines the arrival time at node $j$ coming from the node $i$ and can be 
thought as the amount of flow that passes through the arc $(i,j)$.

The objective function \eqref{eq:CTOP-ST-MINof} maximises the overall profit collected by each route. 
Constraint~\eqref{eq:CTOP-ST-MINstartend} guarantees (together with the sense of the optimisation of the 
objective function) that each route starts from the source node and ends to the destination one.
Constraints~\eqref{eq:CTOP-ST-MINconnectivity} guarantee the connectivity of each route.
Constraints~\eqref{eq:CTOP-ST-MINFlow} and~\eqref{eq:CTOP-ST-MINUB} are the GG subtours elimination 
constraints adapted for TOP. 
Constraints~\eqref{eq:CTOP-ST-MINUB} and~\eqref{eq:CTOP-ST-MINLB} set the upper (lower) bound on the duration of each route.
Constraints~\eqref{eq:CTOP-ST-MINDepot} bound the flow originating from the initial depot.
Constraints~\eqref{eq:CTOP-ST-MINmandatory} ensure that all the mandatory customers $M$ are visited.
Constraints~\eqref{eq:CTOP-ST-MINIncompatibilities} guarantee that all the physical incompatibilities 
are satisfied.
Constraints~\eqref{eq:CTOP-ST-MINDEF1}, \eqref{eq:CTOP-ST-MINDEF2}, \eqref{eq:CTOP-ST-MINDEF3}, 
\eqref{eq:CTOP-ST-MINDEF4} and~\eqref{eq:CTOP-ST-MINDEF5} are variable definition constraints.

In order to include the logical incompatibilities to model the \topPL{}, we can
add the following constraints in the formulation~\eqref{mod:CTOP-ST-MIN}:
\allowdisplaybreaks{
\begin{subequations}
\begin{align}
\quad & v_k \geq j \cdot x_{1k}, \quad \forall \: k \in \hat{N}, \label{eq:CTOP-ST-MINFirstRouteIndex1}\\
\quad & v_k \leq j \cdot x_{1k} - (n - 2) \: (x_{1k} - 1), \quad \forall \: k \in \hat{N}, \label{eq:CTOP-ST-MINFirstRouteIndex2}\\
\quad & v_j \geq v_i + (n - 2) \: (x_{ij} - 1), \quad \forall \: (i,j) \in \hat{A}, \label{eq:CTOP-ST-MINForwardRouteIndex1}\\
\quad & v_j \leq v_i + (n - 2) \: (1 - x_{ij}), \quad \forall \: (i,j) \in \hat{A}, \label{eq:CTOP-ST-MINForwardRouteIndex2}\\
\quad & v_i \geq v_j + 1 - (n - 2) \: (1 - u_{ij}), \quad \forall \: i,j \in C, \, i\neq j \label{eq:CTOP-ST-MINConflicts1}\\
\quad & v_i \leq v_j - 1 + (n - 2) \: u_{ij}, \quad \forall \: i,j \in C, \, i\neq j \label{eq:CTOP-ST-MINConflicts2}\\
\quad & u_{ij} \in \{0, 1\}, \quad \forall \: i,j \in C, \, i\neq j \label{eq:CTOP-ST-MINDEF6}\\
\quad & v_k \in \mathbb{R^+}, \quad \forall \: k \in \hat{N} \, . \label{eq:CTOP-ST-MINDEF7}
\end{align}
\label{mod:CTOP-ST-MIN-logic-incomp}
\end{subequations}
}

The variable $v_k$ \parencite{Furtado2017} represents the index of the first node of the route that
visits node $k$ (the route identifier of the node $k$).
The variable $u_{ij}$ is equal to $1$ if and only if $v_i - v_j \geq 1$, $0$ otherwise.
Constraints~\eqref{eq:CTOP-ST-MINFirstRouteIndex1} and~\eqref{eq:CTOP-ST-MINFirstRouteIndex2} guarantee 
that the route identifier for the node $k$ assumes the same index of the first visited node of the route.
Constraints~\eqref{eq:CTOP-ST-MINForwardRouteIndex1} and~\eqref{eq:CTOP-ST-MINForwardRouteIndex2} guarantee 
that the index of the first visited node is forwarded to the next nodes in the route.
Constraints~\eqref{eq:CTOP-ST-MINConflicts1} and~\eqref{eq:CTOP-ST-MINConflicts2} guarantee that all the 
logical incompatibilities are satisfied: actually, for each pair of logical incompatible nodes $i$ and $j$ in $C$, we 
impose that the route identifier of $i$ must be different from the route identifier of $j$. 

\bigskip

Table~\ref{tab:notation} summarises the notation introduced for the 
formulations~\eqref{mod:MTOP-ST-MIN} and~\eqref{mod:CTOP-ST-MIN} and their extension to include 
the logical incompatibilities~\eqref{mod:MTOP-ST-MIN-logic-incomp} 
and~\eqref{mod:CTOP-ST-MIN-logic-incomp}, respectively.

\begin{table}[!ht]
\centering
\resizebox{0.95\textwidth}{!}{%
\begin{tabular}{llp{24pt}ll}
\toprule
\multicolumn{2}{c}{\textbf{Sets}}   && \multicolumn{2}{c}
{\textbf{Parameters}}\\
\cmidrule{1-2} \cmidrule{4-5}
$N$ & nodes                         && $n$ & number of nodes \\
$\hat{N}$ & customers               && $m$ & number of routes \\
$A$ & arcs                          && $p_k$ & profit associated to the node $k$ \\
$\hat{A}$ & traversable arcs        && $t_{ij}$ & travelling time between the nodes $i$ and $j$ \\ 
$M$ & mandatory nodes               && $s_k$ & service time of the node $k$ \\
$I$ & physical incompatibilities    && $T_{\max}$ & time limit for each route \\
$C$ & logical incompatibilities     \vspace*{3pt}\\
\cmidrule{1-2} \cmidrule{4-5}
\multicolumn{2}{c}{\textbf{Decision variables (Mixed formulation)}} &&
\multicolumn{2}{c}{\textbf{Decision variables (Compact formulation)}} \\
\cmidrule{1-2} \cmidrule{4-5}
$x_{ijr}$ & $1$ if the arc $(i, j)$ is traversed by the route $r$ &&
    $x_{ij}$ & $1$ if the arc $(i, j)$ is traversed \\
$y_{kr}$ & $1$ if the node $k$ is visited by the route $r$ &&
    $y_{k}$ & $1$ if the node $k$ is visited \\
$z_{ijr}$ & flow traversing the arc $(i,j)$ by the route $r$ &&
    $z_{ij}$ & flow traversing the arc $(i,j)$ \\
&&& $v_{k}$ & route identifier of the node $k$ \\
&&& $u_{ij}$ & $1$ if $v_i - v_j \geq 1$ \vspace*{3pt}\\
\bottomrule
\end{tabular}%
}
\caption{Summary of the notation.}
\label{tab:notation}
\end{table}

\section{An exact method}
\label{sec:exact-method}

We developed an exact method based on the cutting-plane methodology that is able to separate five different 
families of valid inequalities to strengthen the model.
%
Our cutting-plane exact algorithm is based on the compact formulation~\eqref{mod:CTOP-ST-MIN}, 
which almost always dominates the mixed formulation~\eqref{mod:MTOP-ST-MIN} from a 
computational point of view, as we shall show in Section~\ref{fc}.

%

\subsection{The valid inequalities}
\label{valid-inequalities}

We introduce the directed graph $\overline{G} = (N, \overline{A})$ with $\overline{A} = \{(i, j) : \overline{x}_{ij} > 0\}$ 
for every optimal relaxed solution ($\overline{x}$, $\overline{y}$, $\overline{z}$) of~\eqref{mod:CTOP-ST-MIN}. Such a graph
makes it easier to describe the valid inequalities. Until the end of this section, we refer to the quantity $\overline{x}_{ij}$ as the amount of flow passing through the arc ($i$, $j$) and to the quantity $\overline{y}_i$ as the amount of flow passing through the node $i$ in the graph $\overline{G}$. 

\subsubsection{Route inequalities}
\label{iri}

A relaxed solution of the compact formulation~\eqref{mod:CTOP-ST-MIN} allows the presence of 
infeasible routes whose travel time exceeds $T_{\max}$. The Route Inequalities (RIs) aims 
to cut off such routes from the search space:
\begin{equation}
    \label{routeIneq}
    \sum_{i=1}^{l - 1} x_{r_ir_{i+1}} \leq \sum_{k=2}^{l - 1} y_{r_k},
\end{equation}
where $r_i$ represents the $i^{th}$ node visited by $r$ and $l$ the cardinality of the latter. 

\subsubsection{Set inequalities}
\label{isi}

The Set Inequalities (SIs) could be considered a strong version of the RIs since they do not take into account the order of the nodes composing a route. 

Thus, given an infeasible route $r$, we build a directed sub-graph $G(r) = (N(r), E(r))$ in which $N(r)$ represents the set of 
nodes composing the route $r$, and $E(r)$ the set of arcs across the nodes in $N(r)$. 
In order to decide if exists a feasible route that visits all the nodes in $N(r)$, we calculate the Helsgaun bound 
\parencite{Helsgaun2000} on $G(r)$ to find a lower bound 
$H_{LB}$ for the travel time of a possible route that visits all the nodes in $N(r)$. 
The Helsgaun bound, as well as the Held-Karp bound \parencite{Held1970}, relies on the Lagrangian relaxation of the 
classical TSP formulation. By consequence, we compute such a bound using the subgradient method setting a simple 
decreasing step size and modifying the Lagrangian multipliers in accordance with \cite{Volgenant1982}.
Note that we decide to adopt the Helsgaun bound instead of the Held-Karp bound since the former is computationally 
less expensive than the latter even though it produces solutions of the same quality.

Finally, if the value of the lower bound $H_{LB}$ is greater than $T_{\max}$, there is no feasible route that visits 
all nodes belonging to $\hat{N}(r) = N(r) \setminus \{1,n \}$, that is the set of customers visited by route $r$. 
In this case, we can consider this inequality:
\begin{equation}
    \label{setIneq}
    \sum_{i \in \hat{N}(r)} \sum_{j \in \hat{N}(r)} x_{ij} \leq |\hat{N}(r)| - 2 \, .
\end{equation}
This cut imposes that the set of nodes in $\hat{N}(r)$ must be visited by, at least, 
two routes. 

\subsubsection{Subpath inequalities}
\label{si}

We adapted the path inequalities \parencite{Fischetti1998} to exploit feasible subpaths in such a way to cut off 
infeasible routes.
Taking a feasible subpath $p$ of customers of cardinality $l$ belonging to an infeasible route 
$r$, we can consider a pair of Sub-Path Inequalities (SPIs):
%
%
\begin{align} 
 \sum_{i=1}^{l - 1} x_{p_ip_{i+1}} - \sum_{k=2}^{l - 1} y_{p_k} - \sum_{v \in L(p)} x_{vp_1} \leq 0 \quad \text{ and } \quad
\sum_{i=1}^{l - 1} x_{p_ip_{i+1}} - \sum_{k=2}^{l - 1} y_{p_k} - \sum_{v \in R(p)} x_{p_{l}v} \leq 0
\label{subpathIneq1}
\end{align}
$L(p)$ and $R(p)$ represent the set of nodes that is possible to add at the beginning 
and at the end of the subpath $p$ without making itself infeasible, respectively. 
Finally, $p_i$ describes the $i^{th}$ node visited by $p$. 

These cuts imposes that the subpath $p$ should be left-connected (right-connected) with 
a node belonging to $L(p)$ ($R(p)$). 
Their validity derives directly from the validity of the path inequalities.

\subsubsection{Subtour elimination constraints}
\label{secs_}

Although the mathematical formulation presented in Section~\ref{mod:CTOP-ST-MIN} is 
sufficient to prevent the generation of subtours and to ensure the connectivity of 
each route, we decided to add a different type of Subtour Elimination Constraints 
(SECs) to strengthen it:
\begin{equation}
\label{eq:secs}
\sum_{i \in U} \sum_{j \in U} x_{ij} \leq \sum_{i \in U} y_i - y_k, \quad \forall \: U \subseteq \{2, \ldots, n-1\}, \; k \in U.
\end{equation}
They are exponential in the number of nodes thus we decided to dynamically separate them along 
the Branch \& Bound tree. Being an adaptation of the classical Dantzig-Fulkerson-Johnson SECs 
(designed for the TSP), the validity of \eqref{eq:secs} follows from \cite{dfj1954}. 
We formulate the following lemma:

\begin{lemma}
\label{lemma}
Any violated SEC \eqref{eq:secs} can only be identified in those sets associated with cycles in the graph 
$\overline{G}$. A subset of nodes $U$ induces a graph in $\overline{G}$ with cycles if and only if the 
arcs involved with the nodes contained in $U$ form one or more cycles in $\overline{G}$.
\end{lemma}

\begin{proof}
To prove it, consider a subset of nodes $U$ associated with no cycles in the graph $\overline{G}$. Now we can consider $\overline G_U$ as the subgraph of $\overline{G}$ induced by the set $U$. 
Until the end of this proof, we refer to $\overline{x}_{ij}$ and $\overline{y}_k$ as the flow passing through the arc ($i, j$) and through the node $k$, respectively.
Thanks to the constraints~\eqref{eq:CTOP-ST-MINconnectivity}, we can observe that for each $k\in U$, the corresponding value $\overline{y}_k$ is at least large as the sum of the flow on the incoming arcs on the node $k$: \[ \sum_{i \in U} \sum_{j \in U} \overline{x}_{ij} \leq \sum_{i \in U} \overline{y}_i .\] 
We define $\theta$ as the flow corresponding to all arcs $(i,j)$ with $i\notin U$ and $j\in U$: \[\theta = \sum_{i \notin U} \sum_{j \in U} \overline{x}_{ij} .\]
Hence, to balance the previous inequality, we need to sum up $\theta$ to the left part of it: \[ \sum_{i \in U} \sum_{j \in U} \overline{x}_{ij} + \theta = \sum_{i \in U} \overline{y}_i .\]
Since $\overline{G}_U$ does not contain any cycle, we can deduce by simple flow conservation that $\theta \geq \overline{y}_k$ for each $k \in U$. Finally, we can conclude that:
\[ \sum_{i \in U} \sum_{j \in U} \overline{x}_{ij} = \sum_{i \in U} \overline{y}_i - \theta \leq \sum_{i \in U} \overline{y}_i - \overline{y}_k \quad k \in U, \]
which completes the proof.
\end{proof}

\begin{figure}[!ht]
\centering
\begin{subfigure}{.40\textwidth}
  \centering
  \begin{tikzpicture}[>=latex, semithick]

\node [circle, draw, semithick, minimum size=0.6cm, label={[label distance=0.05cm]90:$0.9$}] (a) at (3,1) {$a$};
\node [circle, draw, semithick, minimum size=0.6cm, label={[label distance=0.05cm]0:$0.8$}] (b) at (1.5,3) {$b$};
\node [circle, draw, semithick, minimum size=0.6cm, label={[label distance=0.05cm]-90:$0.9$}] (c) at (3,5) {$c$};
\node [circle, draw=none, semithick, minimum size=0.5cm] (f1) at (3,-1) {};
\node [circle, draw=none, semithick, minimum size=0.5cm] (f2) at (-0.5,3) {};
\node [circle, draw=none, semithick, minimum size=0.5cm] (f3) at (3,7) {};
\node [ellipse, draw, dashed, semithick, minimum width=4.5cm, minimum height=5.8cm] (big) at (3,3) {};
\node [rectangle, draw, white, semithick, fill=lightgray, minimum width=0.8cm, minimum height=0.5cm] (fill) at (4.7,3) {};

\draw [->, thick](a) to [bend left] node[left=0.1cm] {0.6} (b);
\draw [->, thick](b) to [bend left] node[left=0.1cm] {0.7} (c);
\draw [->, thick](a) to [bend right=60] node[right=0.1cm] {0.2} (c);
\draw [->, thick](a) to [bend left] node[left=0.1cm] {0.6} (b);
\draw [->, thick](a) to [bend right] node[left=0.1cm] {0.1} (f1);
\draw [->, thick](f1) to [bend right] node[right=0.1cm] {0.9} (a);
\draw [->, thick](b) to [bend right] node[above=0.1cm] {0.1} (f2);
\draw [->, thick](f2) to [bend right] node[below=0.1cm] {0.2} (b);
\draw [->, thick](c) to [bend right] node[right=0.1cm] {0.9} (f3);

\end{tikzpicture}
  \caption{No violated SECs}
  \label{noVCCs}
\end{subfigure}%
\hspace*{.10\textwidth}
\begin{subfigure}{.40\textwidth}
  \centering
  \begin{tikzpicture}[>=latex, semithick]

\node [circle, draw, semithick, minimum size=0.6cm, label={[label distance=0.05cm]90:$0.9$}] (a) at (3,1) {$a$};
\node [circle, draw, semithick, minimum size=0.6cm, label={[label distance=0.05cm]0:$0.8$}] (b) at (1.5,3) {$b$};
\node [circle, draw, semithick, minimum size=0.6cm, label={[label distance=0.05cm]-90:$0.9$}] (c) at (3,5) {$c$};
\node [circle, draw=none, semithick, minimum size=0.5cm] (f1) at (3,-1) {};
\node [circle, draw=none, semithick, minimum size=0.5cm] (f2) at (-0.5,3) {};
\node [circle, draw=none, semithick, minimum size=0.5cm] (f3) at (3,7) {};
\node [ellipse, draw, dashed, semithick, minimum width=4.5cm, minimum height=5.8cm] (big) at (3,3) {};
\node [rectangle, draw, white, semithick, fill=lightgray, minimum width=0.8cm, minimum height=0.5cm] (fill) at (4.7,3) {};

\draw [->, thick](a) to [bend left] node[left=0.1cm] {0.6} (b);
\draw [->, thick](b) to [bend left] node[left=0.1cm] {0.7} (c);
\draw [->, thick](c) to [bend left=60] node[right=0.1cm] {0.9} (a);
\draw [->, thick](a) to [bend left] node[left=0.1cm] {0.6} (b);
\draw [->, thick](a) to [bend right] node[left=0.1cm] {0.3} (f1);
\draw [->, thick](b) to [bend right] node[above=0.1cm] {0.1} (f2);
\draw [->, thick](f2) to [bend right] node[below=0.1cm] {0.2} (b);
\draw [->, thick](f3) to [bend left] node[right=0.1cm] {0.2} (c);

\end{tikzpicture}
  \caption{Some violated SECs}
  \label{VCCs}
\end{subfigure}
\caption{Illustration of the lemma~\ref{lemma}.}
\label{fig:lemma}
\end{figure}
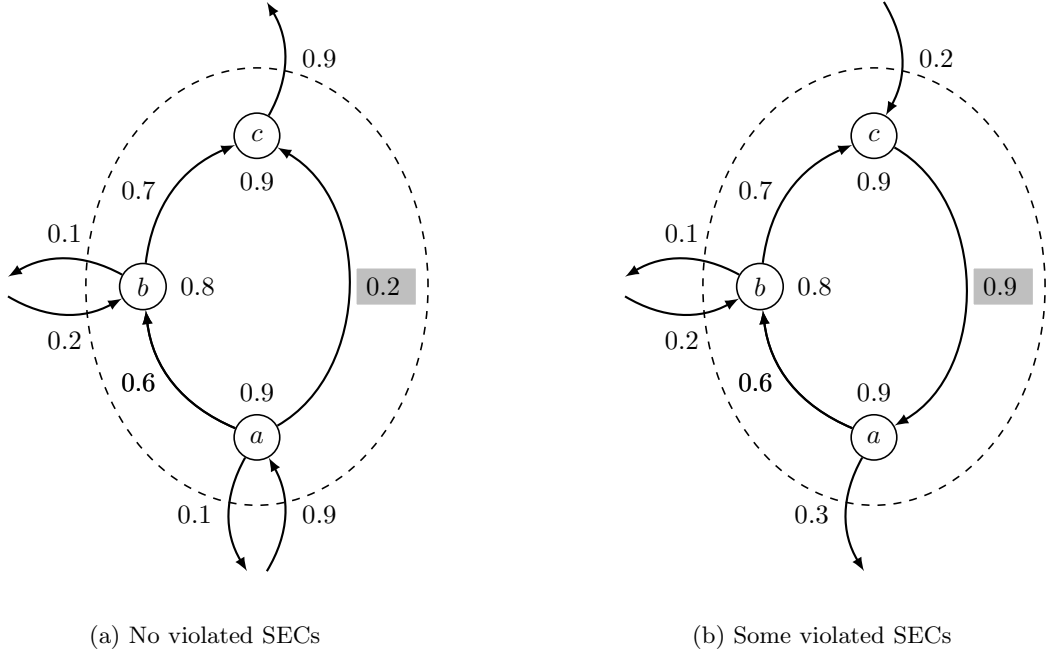

Figure~\ref{fig:lemma} gives a graphic intuition of the Property~\ref{lemma}. 
Figures~\ref{noVCCs} and~\ref{VCCs} depict a portion of the graph $\overline{G}$ that only considers the nodes in $U$, which is
composed of the three nodes $a, b, c$. The numbers close to the nodes and the arcs are the (fractional) values of the flow. The dotted ellipses represent the boundary between $U$ and the rest of the graph, which means that the ingoing flow through the dotted ellipse must be equal to the outgoing flow. 
This property is guaranteed by the constraints~\eqref{eq:CTOP-ST-MINconnectivity}.
It is easy to verify that the lemma~\ref{lemma} holds in Figure~\ref{noVCCs} while not in Figure~\ref{VCCs} due to the presence 
of a cycle. 
We stress the difference between the flow value on the arcs $(a,c)$ in~\ref{noVCCs} and $(c, a)$ in~\ref{VCCs}, with dark grey rectangle.

The application of the lemma~\ref{lemma} allows us to consider only the cycles inside $\overline{G}$ instead of taking into account the entire graph. Hence, we can simply separate all the violated SECs by using an algorithm for 
the elementary cycles detection \parencite{Hawick2008}. Asymptotically, this separation method results to be faster 
than the classical one used in literature that runs the max-flow/min-cut algorithm 
(see, e.g., the algorithm in \cite{boykov2004}) for each node of the graph. 
Practically, we separate them in $\mathcal{O}\left(\left(|N| + |A|\right) \cdot \mu \right)$ instead of $\mathcal{O}\left(|N|^3 \cdot |A| \cdot \tau \right)$ 
where, $\mu$ represents the number of cycles inside $\overline{G}$ and $\tau$, the cost of the maximum flow/minimum cut 
for the same graph.

\subsubsection{Logical inequalities}
\label{lii}

The mathematical formulation \ref{mod:CTOP-ST-MIN} allows the presence of routes that contain logical incompatibilities 
in a relaxed solution. The Logical Inequalities (LIs) have the same structure of the SIs ones presented in Section~\ref{isi},
and their aims is to cut off every route composed by a pair of logically incompatible nodes, at least. 
For a given route $r$, we identify all the subpaths $p$ of customers belonging to $r$ checking 
if the starting and ending nodes of $p$ are logical incompatible. 
In this case, we can consider this inequality:
\begin{equation}
    \label{logicalIneq}
    \sum_{i \in \hat{N}(p)} \sum_{j \in \hat{N}(p)} x_{ij} \leq |\hat{N}(p)| - 2 .
\end{equation}
Even in this case, this cut imposes that the nodes contained in $\hat{N}(p)$ must be visited by 
two routes at least. 



\subsection{A cutting-plane algorithm}
\label{cp}

We propose an exact algorithm for solving the TOP-ST-MIN based on the cutting-plane methodology 
that adopts all the valid inequalities barely introduced to strengthen the linear programming 
formulation. We refer to such an algorithm as CPA.

We recall that the classic cutting-plane method iteratively refines the feasible region by 
adding valid linear inequalities obtained by solving the corresponding separation problem. 
Such a valid inequalities reduces the search space. This process is repeated until an 
optimal integer solution is found. 
%
%
For every optimal relaxed solution, we build the graph $\overline{G}$, and we calculate 
the entire set $R$ of routes contained in $\overline{G}$ with a modified version of the 
Depth-First Search (DFS) algorithm. 

\paragraph*{Valid Inequalities}
First, we consider are the \emph{Subpath inequalities} (Section~\ref{si}). To separate them, 
we take into account all the 
infeasible routes $r \in R$ $\left(\delta_r > T_{\max}\right)$ where, $\delta_r$ represents the sum of 
travel and service times of the arcs and nodes composing $r$. For each feasible subpath $p$ of 
customers belonging to $r$, 
we create the associated subroute $s$ adding the source (destination) node at the beginning (end) of $p$. 
If $\delta_s$ results to be less or equals $T_{\max}$, we build the sets $L(p)$ and $R(p)$ and we check if the 
associated inequalities are violated. 

After that, we consider the \emph{Set inequalities} (Section~\ref{isi}) and \emph{Route inequalities} (Section~\ref{iri}). 
To separate the first one, we calculate the Helsgaun lower bound $H_{LB}$ on the sub-graph $G(r)$ composed by the nodes 
in $r$ and, if $H_{LB}$ results to be greater than $T_{\max}$, the first inequality is checked for violation. Otherwise, 
we check if the second one is violated. 
Furthermore, we also check the \emph{Logical inequalities} (Section~\ref{lii}) for violation when we are solving the \topPL{}.

Finally, we consider all the elementary cycles inside $\overline{G}$ and we check if the associated \emph{subtours 
Constraints} (Section~\ref{secs_}) are violated.

\paragraph*{Preprocessing}
Some instances can contains nodes and arcs that are unreachable due to the restricted time budget and to the incompatibilities.
A node or an arcs is unreachable if and only if it makes the route infeasible as soon as it is included in the route. Except for the physical incompatibilities, that we already consider in the formulation, we define the sets $U_n$ and $U_e$ as the sets are the unreachable set of nodes and the unreachable set of arcs. Both are filled keeping into account the total route time and the logical incompatibilities. For the first case, a node (an arc) is considered to be unreachable if and only if, including it in a route, its total travel time exceeds $T_{\max}$. In the second case instead, we simply consider unreachable each arc $(i,j)$ in which the nodes $i$ and $j$ are logical incompatible. If the two sets are not empty, we can fix the corresponding variables as follows:
\begin{align} 
y_{k} \leq 0, \quad \forall \: k \in U_n \quad \text{ and } \quad
x_{ij} \leq 0, \quad \forall \: (i,j) \in U_e.
\end{align}

\paragraph*{Implementation}

The pseudocode of the CPA callback is reported in Algorithm \ref{algo:bc} and makes use of the following procedures:
%
\begin{inparaenum}[(i)]
  \item \texttt{buildGraph} builds the graph $\overline{G}$ from the relaxed solution;
  \item \texttt{buildRoutes} builds the set $R$ of routes from the graph $\overline{G}$ with a variant of the DFS algorithm;
  \item \texttt{buildSubRoute} builds the corresponding subroute from the subpath $p$;
  \item \texttt{buildSubGraph} builds the corresponding the sub-graph $G(r)$ from the nodes contained in the route $r$;
  \item \texttt{calculateTravelTimeBound} calculate the Helsgaun lower bound for the sub-graph $G(r)$;
  \item \texttt{buildSets} builds the corresponding sets $L(p)$ and $R(p)$ from the subpath $p$.
\end{inparaenum}
We developed the overall CPA callback as an embedded procedure inside the CPLEX 
environment, which is called for every branching node with an optimal fractional 
solution.

\allowdisplaybreaks{
\begin{algorithm}[!ht]
\KwData{$N$, $T_{\max}$, $t_{ij}$, $\overline{x}_{ij}$ with $i,j \in \{1, \dots, n\}$, $s_k$, $\overline{y}_k$ with $k \in \{2, \dots, n - 1\}$}
    $\overline{G} \leftarrow$ buildGraph($\overline{x}_{ij}$, $\overline{y}_k$) \;
    $R \leftarrow$ buildRoutes($\overline{G}$, $1$, $n$) \;
    
    \For{route $r \in R$}{
        \If{$\delta_r > T_{\max}$}{
            \For{subpath $p \in r$}{
                \If{$1 \notin p$ and $n \notin p$}{
                    $s \leftarrow$ buildSubRoute($p$) \;
                    \If{$\delta_s \leq T_{\max}$}{
                        $L(p), R(p) \leftarrow$ buildSets($p$, $N$, $T_{\max}$) \;
                        add inequalities \eqref{subpathIneq1} associated with $p$, $L(p)$ and $R(p)$ \;
                    }
                }
            }

            $G(r) \leftarrow$ buildSubGraph($r$, $t_{ij}$, $s_k$) \;
            $H_{LB} \leftarrow$ calculateTravelTimeBound($G(r)$) \;
    
            \If{$H_{LB} > T_{\max}$}{
                add inequality \eqref{setIneq} associated with $\hat{N}(r)$ \;
            }
            \Else
            {
                add inequality \eqref{routeIneq} associated with $r$ \;
            }
        }

        \For{subpath $p \in r$}{
            \lIf{$1 \notin p$ and $n \notin p$}{
                add inequality \eqref{logicalIneq} associated with $p$
            }
        }
    }

    \For{cycle $c \in \overline{G}$}{
        add inequalities \eqref{eq:secs} associated with $c$ \;
    }

\caption{The \texttt{CPA} callback.} 
\label{algo:bc}
\end{algorithm}
}

\section{The instances generation}
\label{sec:instances}

We generated new instances for the TOP-ST-MIN in such a way to evaluate and validate the 
formulations proposed in Section~\ref{sec:models} and the algorithm described in 
Section~\ref{cp}. The main objective is to generate instances in order to highlight 
the computational impact of the three new features of the TOP-ST-MIN.

As a starting point, we used the TOP benchmark sets available in the 
literature\footnote{The TOP instances are available 
at \url{https://www.mech.kuleuven.be/en/cib/op}}. 
This benchmark set, proposed by \cite{CHAO1996464}, is composed of 387 instances divided 
into 7 sets in accordance with the number of 
nodes ranging from 21 to 102. Each set differs from the number of teams (vehicles, from 2 to 4) 
available to serve the customers and for the value of the maximum time duration of a route. 
The instances are composed by nodes represented though two-dimensional coordinates in a 
Euclidean plane.

%

In order to evaluate the computational impact of the new features, we considered different 
logics to generate the new instances.
In the following, we describe two different generation methods for the mandatory nodes, two 
for the physical incompatibilities, two for the 
logical incompatibilities, and one generation method for the service times. 

We would remark that all of them (except for the service times generation method) are developed 
in a fully-deterministic mode in order to 
control the number and the choice of the mandatory nodes and to decide a specific number and topology of the arcs inside the graph. 
In order to limit the possibility of generating an infeasible instance, we guarantee the existence 
of routes that can visit each single mandatory node without 
violating any incompatibility, and holding the time budget constraint. All the chosen parameters 
in each method have been identified after some preliminary tests.

\subsection{The mandatory nodes generation}
We identified two different methods to generate the mandatory nodes. The former (Scattered Mandatory (SM)) selects the mandatory nodes in such a way to maximise the sum of distances between all the pairs of selected nodes while the latter (Clustered Mandatory (CM)) selects the mandatory nodes in the opposite way, that is minimising the sum of distances between all the pairs of selected nodes.
To do this, we adapted the algorithms recently reported in \cite{Aringhieri2023} for the Maximum Diversity Problem \parencite{Kuo1993}.
We set as mandatory $5\%$ of the total number of customers as \cite{Assunção2020} for the 
Steiner TOP.

\subsection{The physical incompatibilities generation}
We adopt two strategies to define the topology of the arcs composing the graph 
representing the instances. 
The first one (the Clusters-based Physical Incompatibilities (CPI) method) tends to create 
a graph composed of subgraphs in which some of them are completely or partially disconnected 
from each other.

On the contrary, the second one (the Degree-based Physical Incompatibilities (DPI) method) tends 
to generate a graph in which the degree of each node is similar to each other. 
To avoid generating overly simple or infeasible instances, we remove $20\%$ of the total number 
of arcs (practically, we are imposing $20\%$ of physical incompatibilities using one of the 
proposed approaches). 
Finally, we also guarantee a bidirectional crossing of the arcs, i.e.
if $(i,j) \in \hat{A}$ then $(j,i) \in \hat{A}$ or
if $(i,j) \notin \hat{A}$ then $(j,i) \notin \hat{A}$, and vice versa. 

\subsubsection*{The clusters-based approach}

The basic idea is to build a graph by partitioning the nodes and reducing the connections among
nodes belonging to different partitions.
We start by running the K-Means algorithm (\cite{MacQueen1967SomeMF}) to partition 
all the nodes (except the source and destination ones) in $c$ different clusters 
(we choose $c = 3$). It is worth noting that such an algorithm may not create
balanced partitions.
Then, we impose an incompatibility between pairs of clusters (say, e.g., clusters $k$ and $w$) through a uniform
distribution, meaning that all the nodes belonging to cluster $k$ could end up being 
incompatible with the ones belonging to cluster $w$.
The final graph is built by solving an integer linear program which selects a given
amount of arcs to be added in $\hat{A}$ in such a way to minimise the number 
of selected arcs among pairs of incompatible clusters.
As a secondary objective, the model tries to balance the number of arcs inside
a given cluster and towards other (compatible) clusters. The idea is to foster
the generation of instances with a larger number of feasible solutions.
The mathematical model follows.

\allowdisplaybreaks{
\begin{subequations}
\begin{align}
\min \quad & (\gamma + 1) \left(\sum_{i=1}^{n} \sum_{j=1}^{n} \sum_{k=1}^{c} \sum_{w=1}^{c} \lambda_{ik} \: 
\lambda_{jw} \: \phi_{kw} \: x_{ij} \right) + \sigma \label{eq:min_incs}\\
\textrm{s.t.} \quad & \sum_{i=1}^{n} \sum_{j=1}^{n} \sum_{k=1}^{c} \sum_{w=1}^{c} \lambda_{ik} \: 
\lambda_{jw} \: x_{ij} = \gamma, 
\label{eq:cl_edges}\\
\quad & \sum_{i=1}^{n} \sum_{j=1}^{n} \lambda_{ik} \: \lambda_{jk} \: x_{ij} \leq \sigma \quad \forall \: k = 1, \ldots, c,
\label{eq:incs_intra}\\
\quad & \sum_{i=1}^{n} \sum_{j=1}^{n} \sum_{w=1}^{c} \lambda_{ik} \: \lambda_{jw} \: x_{ij} \leq \sigma \quad \forall \: k = 1, \ldots, c,
\label{eq:incs_inter}\\
\quad & x_{ij} = x_{ji}, \quad \forall \: i,j = 1, \ldots, n,
\label{eq:cl_bidirectional}\\
\quad & x_{ij} \in \{0, 1\}, \quad \forall \: i,j = 1, \ldots, n, \label{eq:cl_def1}\\
\quad & \sigma \in \mathbb{N^+} \label{eq:cl_def2}
\end{align}
\end{subequations}
}

The parameters $\lambda_{ik}$ is equal to $1$ if and only if the node $i$ belongs to the cluster $k$, $0$ otherwise, and $\phi_{kw}$ is 
equal to $1$ if and only if the cluster $k$ is incompatible with the cluster $w$, $0$ otherwise. 
The variable $x_{ij}$ is $1$ if the arc $(i,j)$ has been chosen, $0$ otherwise. The variable $\sigma$ counts the maximal number of arcs 
intra and inter clusters. 

The hierarchical objective function~\eqref{eq:min_incs} minimises the arcs that violate the 
physical incompatibilities as first goal balancing the intra 
and inter arcs between clusters as a second goal. The constraint~\eqref{eq:cl_edges} guarantees that exactly $\gamma$ ($\gamma$ is equal to 80\% of $|\hat{A}|$) arcs will be chosen. 
The constraints~\eqref{eq:incs_intra} and~\eqref{eq:incs_inter} calculate the maximal number of arcs intra and inter clusters, 
respectively. The constraints~\eqref{eq:cl_bidirectional} guarantee the bidirectional crossing of the arcs. 
Finally, the constraints~\eqref{eq:cl_def1} and~\eqref{eq:cl_def2} are the variable definition constraints.

\subsubsection*{The degree-based approach}

Compared to the previous approach, the aim is now to select the arcs in such a way that 
the degree of the nodes is similar in such a way to balance the incompatibilities among 
the nodes. To create a graph along these lines, we developed the following integer 
linear program:

\allowdisplaybreaks{
\begin{subequations}
\begin{align}
\min \quad & \alpha \label{eq:max_degree}\\
\textrm{s.t.} \quad & \sum_{i=1}^{n} \sum_{j=1}^{n} \: x_{ij} = \gamma, 
\label{eq:deg_edges}\\
\quad & \sum_{k=1}^{n} \: x_{ik} \leq \alpha \quad \forall \: i = 1, \ldots, n,
\label{eq:degree}\\
\quad & x_{ij} = x_{ji}, \quad \forall \: i,j = 1, \ldots, n,
\label{eq:deg_bidirectional}\\
\quad & x_{ij} \in \{0, 1\}, \quad \forall i,j = 1, \ldots, n,
\label{eq:deg_def1}\\
\quad & \alpha, \beta \in \mathbb{N^+}
\label{eq:deg_def2}
\end{align}
\end{subequations}
}

The objective function~\eqref{eq:max_degree} minimises the maximum degree between 
all vertices.
Together with the constraint~\eqref{eq:deg_edges}, which guarantees that exactly 
$\gamma$ arcs will be chosen, and the constraints~\eqref{eq:degree}, which links
variables $\alpha$ and $x_{ij}$, the objective function fosters the balancing of 
the degree among the nodes.
The constraints~\eqref{eq:deg_bidirectional} guarantee the bidirectional crossing of the arcs.
Finally, the constraints~\eqref{eq:deg_def1} and~\eqref{eq:deg_def2} are the variable 
definition constraints.

\subsection{The logical incompatibilities generation}

We identified two different methods to generate the logical incompatibilities. 
The former (Farthest Logical Incompatibilities (FLI)) imposes a logical incompatibility 
between the node $i$ and every node included in the set of farthest nodes to $i$. On the 
contrary, the latter (Nearest Logical Incompatibilities (NLI)) imposes a logical incompatibility 
between the node $i$ and every node included in the set of nearest nodes to $i$. The set of 
farthest and nearest nodes are calculated by taking $5\%$ of the farthest and nearest nodes 
to $i$, respectively.

The former method favours routes where the customers tend to be close to each other while the
latter favours routes where the customers are distant from one another.

\subsection{The service times generation}

Finally, the last new feature we need to generate is the service times of the customers,
which are generated using a uniform distribution. Essentially, we consider a global amount 
of service time ($\left(m \: T_{\max}\right) \: / \: 2$) to be randomly distributed among 
the nodes. 
Further, in order to limit the generation of infeasible instances, we increase the time 
limit for each route by $T_{\max} \: / \: 2$.


\subsection{The instance generation scheme}

Combining all the presented methods, we developed 4 and 8  
different generation schemes respectively for the \topP{} and \topPL{} instances, 
as reported in Figure~\ref{fig:genschemes}: the root node represents a generic 
initial TOP instance while each level of the tree consists in the generation of a
specific new feature. 
Assuming that the root of this picture represents the level 0, a path from the 
root to a node at the level 2  (dark grey area) describes a generation scheme for a 
\topP{} instance while a path from the root to a leaf (light grey area) describes 
a generation scheme for a \topPL{} instance.

\begin{figure}[!ht]
\centering
\begin{adjustbox}{width=1\textwidth,center}
\begin{tikzpicture}[>=latex, semithick]

\node [rectangle, draw, black, semithick, fill={black!5}, minimum width=16.5cm, minimum height=6.5cm,rounded corners=.55cm] (greenbg) at (8,3.5) {};
\node [rectangle, draw, black, semithick, fill={black!20}, minimum width=14cm, minimum height=4.5cm,rounded corners=.55cm] (bluebg) at (8,4) {};

\node [ellipse, draw, semithick, minimum width=1.5cm, minimum height=0.8cm, fill=white] (TOP) at (8,5.5) {TOP};
\node [ellipse, draw, semithick, minimum width=1.5cm, minimum height=0.8cm, fill=white] (CM) at (4,4) {CM};
\node [ellipse, draw, semithick, minimum width=1.5cm, minimum height=0.8cm, fill=white] (SM) at (12,4) {SM};
\node [ellipse, draw, semithick, minimum width=1.5cm, minimum height=0.8cm, fill=white] (CPI1) at (2,2.5) {CPI};
\node [ellipse, draw, semithick, minimum width=1.5cm, minimum height=0.8cm, fill=white] (DPI1) at (6,2.5) {DPI};
\node [ellipse, draw, semithick, minimum width=1.5cm, minimum height=0.8cm, fill=white] (CPI2) at (10,2.5) {CPI};
\node [ellipse, draw, semithick, minimum width=1.5cm, minimum height=0.8cm, fill=white] (DPI2) at (14,2.5) {DPI};
\node [ellipse, draw, semithick, minimum width=1.5cm, minimum height=0.8cm, fill=white] (FLI1) at (1,1) {FLI};
\node [ellipse, draw, semithick, minimum width=1.5cm, minimum height=0.8cm, fill=white] (NLI1) at (3,1) {NLI};
\node [ellipse, draw, semithick, minimum width=1.5cm, minimum height=0.8cm, fill=white] (FLI2) at (5,1) {FLI};
\node [ellipse, draw, semithick, minimum width=1.5cm, minimum height=0.8cm, fill=white] (NLI2) at (7,1) {NLI};
\node [ellipse, draw, semithick, minimum width=1.5cm, minimum height=0.8cm, fill=white] (FLI3) at (9,1) {FLI};
\node [ellipse, draw, semithick, minimum width=1.5cm, minimum height=0.8cm, fill=white] (NLI3) at (11,1) {NLI};
\node [ellipse, draw, semithick, minimum width=1.5cm, minimum height=0.8cm, fill=white] (FLI4) at (13,1) {FLI};
\node [ellipse, draw, semithick, minimum width=1.5cm, minimum height=0.8cm, fill=white] (NLI4) at (15,1) {NLI};

\draw [->, thick](TOP) to (CM);
\draw [->, thick](TOP) to (SM);
\draw [->, thick](CM) to (CPI1);
\draw [->, thick](CM) to (DPI1);
\draw [->, thick](SM) to (CPI2);
\draw [->, thick](SM) to (DPI2);
\draw [->, thick](CPI1) to (FLI1);
\draw [->, thick](CPI1) to (NLI1);
\draw [->, thick](DPI1) to (FLI2);
\draw [->, thick](DPI1) to (NLI2);
\draw [->, thick](CPI2) to (FLI3);
\draw [->, thick](CPI2) to (NLI3);
\draw [->, thick](DPI2) to (FLI4);
\draw [->, thick](DPI2) to (NLI4);

\end{tikzpicture}
\end{adjustbox}
\caption{Illustration of the instances generation schemes.}
\label{fig:genschemes}
\end{figure}
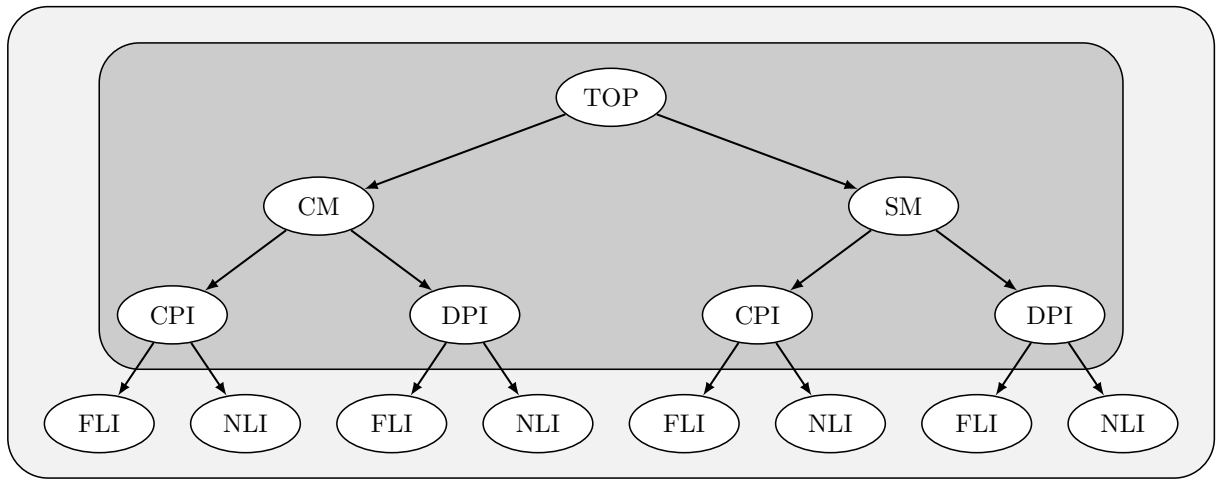

Table~\ref{instances} summarises the results of the generation process. 
Since the number of TOP-ST-MIN instances would have been too large, we considered a subset of the instances belonging to the initial TOP benchmark set.
The columns TOP report the information about the TOP instances (\cite{CHAO1996464}), that is 
the number of instances taken into account for the generation process, and the number of nodes 
and arcs for each instance in the initial benchmark set.
The columns \topP{} report the number of new instances generated, and the number of mandatory 
nodes ($|M|$) and physical incompatibilities ($|I|$) while the columns \topPL{} add the
number of logical incompatibilities ($|C|$). 
Note that all the initial instances are generated on the two-dimensional euclidean plane and, 
as a consequence, the classical triangle inequality property holds for each new generated instance. 

Finally, we grouped the new generated instances in three benchmark sets based on the 
number of nodes (\emph{small}, \emph{medium} and \emph{large}) as reported in 
Table~\ref{instances}.

\begin{table}[!ht]
\centering
\begin{tabular}{@{}lrrrrrrrrrrr@{}}
\toprule
                & \multicolumn{1}{l}{\textbf{}}    & \multicolumn{3}{c}{\textbf{TOP}}                                                                      & \multicolumn{3}{c}{\textbf{TOP-ST-MIN-P}}                                                             & \multicolumn{4}{c}{\textbf{TOP-ST-MIN-PL}}                                                                                               \vspace*{1pt} \\
                & \textbf{Set} & \textbf{\#} & \textbf{$|N|$} & \textbf{$|A|$} & \textbf{\#} & \textbf{$|M|$} & \textbf{$|I|$} & \textbf{\#} & \textbf{$|M|$} & \textbf{$|I|$} & \textbf{$|C|$} \\
                 \cmidrule(r){1-2} \cmidrule(lr){3-5} \cmidrule(lr){6-8} \cmidrule(l){9-12}
\textbf{SMALL}  & 1                                & 15                              & 32                               & 992                              & 60                              & 2                                & 198                              & 120                             & 2                                & 198                              & 30                               \\
\textbf{}       & 2                                & 9                               & 21                               & 420                              & 36                              & 1                                & 84                               & 72                              & 1                                & 84                               & 10                               \\
\textbf{}       & 3                                & 15                              & 33                               & 1056                             & 60                              & 2                                & 211                              & 120                             & 2                                & 211                              & 30                               \\
\textbf{}       & \textbf{ALL}                     & 39                              & -                                & -                                & 156                             & -                                & -                                & 312                             & -                                & -                                & -                                \\
\cmidrule(r){1-2} \cmidrule(lr){3-5} \cmidrule(lr){6-8} \cmidrule(l){9-12}
\textbf{MEDIUM} & 5                                & 15                              & 66                               & 4290                             & 60                              & 3                                & 858                              & 120                             & 3                                & 858                              & 86                               \\
\textbf{}       & 6                                & 15                              & 64                               & 4032                             & 60                              & 3                                & 806                              & 120                             & 3                                & 806                              & 100                              \\
\textbf{}       & \textbf{ALL}                     & 30                              & -                                & -                                & 120                             & -                                & -                                & 240                             & -                                & -                                & -                                \\
\cmidrule(r){1-2} \cmidrule(lr){3-5} \cmidrule(lr){6-8} \cmidrule(l){9-12}
\textbf{LARGE}  & 4                                & 15                              & 100                              & 9900                             & 60                              & 5                                & 1980                             & 120                             & 5                                & 1980                             & 230                              \\
& 7                                & 15                              & 102                              & 10302                            & 60                              & 5                                & 2060                             & 120                             & 5                                & 2060                             & 240                              \\
                & \textbf{ALL}                     & 30                              & -                                & -                                & 120                             & -                                & -                                & 240                             & -                                & -                                & -                                \\ \bottomrule
\end{tabular}
\caption{A summary of the newly generated instances.}
\label{instances}
\end{table}

\section{Quantitative analysis}
\label{sec:analysis}

In this section, we report an extensive quantitative analysis in order to validate our 
approach and to provide some insights concerning the structures of our newly generated 
instances.
We will report the results both for \topP{} and \topPL{}.
First, we provide a computational comparison of the two formulations presented in Section~\ref{sec:models} (Sect~\ref{fc}).
Then we evaluate the impact of the valid inequalities discussed in 
Section~\ref{valid-inequalities} (Sect.~\ref{vic}).
The results of the CPA proposed in Section~\ref{cp} are reported in Section~\ref{results}.
Finally, the results of an adaptation of the CPA to solve the TOP is compared with the 
best known results of the literature (Sect.~\ref{results-top}).




The CPA has been implemented in C++ by using CPLEX 22.1.1 Concert Technology. The 
code was compiled in Scientific Linux 7.9 (Nitrogen). The experiments were carried 
out on a 64-bit Linux machine, with a Six-Core AMD Opteron(tm) Processor 8425 HE, 
2.10 GHz, and 24 GB of RAM. CPLEX built-in cuts have been used in all experiments. 
We run CPLEX in a fully-deterministic mode in order to be able to fully reproduce 
the results. All the computational experiments have been run by setting a time 
limit of $7200$ seconds as many authors in the literature 
(see, e.g., \cite{Hanafi2020} and \cite{Assunção2020}).

\subsection{Comparing mixed and compact formulations}
\label{fc}

The main purpose of this section is to compare the mixed formulation 
(Sect.~\ref{ssec:MTOP-ST-MIN}) and the compact formulation 
(Sect.~\ref{ssec:CTOP-ST-MIN}) from a computational point 
of view in order to validate the choice of using the compact formulation in our solution approach. 
The comparisons have been obtained by using CPLEX as a general purpose solver. 


Table~\ref{PTOP-ST-MINPrel} reports a brief comparison between the two models for the \topP{}
(formulations~\eqref{mod:MTOP-ST-MIN} and~\eqref{mod:CTOP-ST-MIN})
considering the 156 instances belonging to the small benchmark set.
We classify the results by number of available vehicles $m$ (from 2 to 4) to analyse the 
performances of both models. 
The column OPT reports the number of instances optimally solved, CPU(s) the average CPU time 
(in seconds) used by CPLEX to certify the optimality, and NODES the average number of branching nodes. 

\begin{table}[!ht]
\centering
\begin{adjustbox}{width=1\textwidth,center}
\begin{tabular}{@{}ccrrrrrrrr@{}}
\toprule
\textbf{} & \textbf{} & \multicolumn{3}{c}{\textbf{Mixed formulation}} & \multicolumn{3}{c}{\textbf{Compact formulation}} & \multicolumn{2}{c}{\textbf{Gaps}} \vspace*{1pt} \\
\textbf{$m$} & \textbf{\#} & \textbf{OPT} & \textbf{CPU(s)} & \textbf{NODES} & \textbf{OPT} & \textbf{CPU(s)} & \textbf{NODES} & \textbf{CPU} & \textbf{NODES} \\ \cmidrule(r){1-2} \cmidrule(lr){3-5} \cmidrule(lr){6-8} \cmidrule(l){9-10}
2 & 52 & 52 & 203.77 & 10520 & 52 & 40.76 & 6403 & -399.89\% & -64.29\% \\
3 & 52 & 52 & 505.52 & 14234 & 52 & 25.72 & 3376 & -1865.45\% & -321.57\% \\
4 & 52 & 51 & 348.62 & 6848 & 52 & 4.94 & 1066 & -6955.07\% & -542.63\% \\ \bottomrule
\end{tabular}
\end{adjustbox}
\caption{The comparison between the mixed and the compact formulations for \topP{}.}
\label{PTOP-ST-MINPrel}
\end{table}

The results reported in Table~\ref{PTOP-ST-MINPrel} show that the compact formulation dominates 
the mixed formulation when solving the \topP{}. In fact, even if the number of optimally solved 
instances is almost the same, the gaps in terms of running time and explored nodes are in favour
of the compact formulation and tend to grow as soon as the number of routes increases.

\begin{table}[!ht]
\centering
\begin{adjustbox}{width=1\textwidth,center}
\begin{tabular}{@{}ccrrrrrrrr@{}}
\toprule
 &  & \multicolumn{3}{c}{\textbf{Mixed formulation}} & \multicolumn{3}{c}{\textbf{Compact formulation}} & \multicolumn{2}{c}{\textbf{Gaps}} \vspace*{1pt} \\
\textbf{$m$} & \textbf{\#} & \textbf{OPT} & \textbf{CPU(s)} & \textbf{NODES} & \textbf{OPT} & \textbf{CPU(s)} & \textbf{NODES} & \textbf{CPU} & \textbf{NODES} \\  \cmidrule(r){1-2} \cmidrule(lr){3-5} \cmidrule(lr){6-8} \cmidrule(l){9-10}
2 & 104 & 103 & 118.67 & 4665 & 97 & 583.76 & 41467 & 79.67\% & 88.75\% \\
3 & 104 & 103 & 249.10 & 7224 & 103 & 160.61 & 10973 & -55.10\% & 34.17\% \\
4 & 104 & 104 & 229.79 & 5262 & 104 & 51.79 & 4838 & -343.75\% & -8.78\% \\ \bottomrule
\end{tabular}
\end{adjustbox}
\caption{The comparison between the mixed and the compact formulations for \topPL{}.}
\label{PLTOP-ST-MINPrel}
\end{table}

Table \ref{PLTOP-ST-MINPrel} reports the results of the comparison for the \topPL{}
(formulations \eqref{mod:MTOP-ST-MIN}+\eqref{mod:MTOP-ST-MIN-logic-incomp}, 
and \eqref{mod:CTOP-ST-MIN}+\eqref{mod:CTOP-ST-MIN-logic-incomp})
for which the dominance is less evident than those for the \topP{}.
As a matter of fact, the mixed formulation obtains better results for instances with only 2 
available vehicles ($m=2$). 
On the contrary, the compact formulation produces better results as soon as the number 
of routes increases. This behaviour may be explained by the larger number of variables 
in the mixed formulation due to the presence of three-index variables $x_{ijr}$.

\subsection{Evaluating the impact of the valid inequalities}
\label{vic}

Prior to analysing the overall results of the CPA, an evaluation was conducted 
to establish the impact of the valid inequalities -- reported in Section~\ref{valid-inequalities} 
-- in the solution process.
The tests were carried out on the small benchmark set.
The column OPT reports the number of instances optimally solved, CPU(s) the 
average CPU time (in seconds) spent by the algorithm to certify the optimality, NODES the average number 
of explored branching nodes, and GAP(\%) the average mip gap of the unsolved instances. 
The first five rows report the results obtained using a single family of 
inequalities while the last two rows report respectively the results using 
all inequalities and none (CPLEX with only bbuilt-in cuts).

\begin{table}[!ht]
\centering
\begin{adjustbox}{width=1\textwidth,center}
\begin{tabular}{@{}lrrrrrrrr@{}}
\toprule
& \multicolumn{4}{c}{\textbf{\topP{}}} & \multicolumn{4}{c}{\textbf{\topPL{}}} \\
\textbf{Valid inequalities} & 
\multicolumn{1}{c}{\textbf{OPT}} & \multicolumn{1}{c}{\textbf{CPU(s)}} & \multicolumn{1}{c}{\textbf{NODES}} & \multicolumn{1}{c}{\textbf{GAP(\%)}} &
\multicolumn{1}{c}{\textbf{OPT}} & \multicolumn{1}{c}{\textbf{CPU(s)}} & \multicolumn{1}{c}{\textbf{NODES}} & \multicolumn{1}{c}{\textbf{GAP(\%)}} \\ 
\cmidrule(r){1-1} \cmidrule(rl){2-5} \cmidrule(l){6-9}
RIs \eqref{iri} & 156 & 24.11 & 3219 & - & 302 & 416.11 & 23957 & 7.73 \\
SIs \eqref{isi} & 156 & 26.97 & 3556 & - & 302 & 440.03 & 28224 & 6.83 \\
SPIs \eqref{si} & 156 & 26.45 & 3362 & - & 305 & 386.63 & 25314 & 5.09 \\
SECs \eqref{secs_} & 156 & 18.11 & 1639 & - & 302 & 446.78 & 23206 & 8.48 \\
LIs \eqref{lii} & -- & -- & -- & -- & 312 & 77.47 & 5287 & - \\
CPA \eqref{cp} & 156 & 17.50 & 1765 & - & 312 & 59.61 & 3254 & - \\
NONE (CPLEX) & 156 & 24.16 & 3683 & - & 304 & 257.01 & 18511 & 5.88
 \\ \bottomrule
\end{tabular}
\end{adjustbox}
\caption{The impact of the valid inequalities for the \topP{} and the \topPL{}.}
\label{TOP-ST-MINCutsImpact}
\end{table}

The results reported in Table~\ref{TOP-ST-MINCutsImpact} clearly show that the separation 
scheme adopted in the CPA produces the best results in terms of running time and 
number of explored nodes. 
On the contrary, if we consider each valid inequality singularly, the SECs (Sect.~\ref{secs_}) 
for the \topP{} and the Logical inequalities (Sect.~\ref{lii}) for the \topPL{} seems to perform 
quite well alone. 
In both cases, these valid inequalities are those that notably reduce the number of 
branching nodes explored. 

It is worth noting that all the remaining valid inequalities -- when used singularly -- increase 
the number of explored nodes compared to CPLEX with only built-in cuts. 
Furthermore, the logical inequalities are the only inequalities that actually allow to solve all 
the \topPL{} instances to optimality.

\subsection{Results on new benchmark instances}
\label{results}

The main purpose of this section is to report the extensive computational analysis of the 
CPA on the new benchmark instance sets, which are generated as discussed in 
Section~\ref{sec:instances}.
In the following, we first report the results for \topP{} and then for \topPL{.}

Table~\ref{PTOP-ST-MIN-BC} reports the results for the \topP{.}
We classify the results into several different solution categories.
The columns OPT describe the statistics related to the instances optimally solved.
The columns NO OPT describe the statistics related to the instances for which a solution has 
been found without proving its optimality.
The columns INFS describe the statistics related to the instances proven to be infeasible. 
Finally, the columns NO SOLS describe the statistics related to the instances for which no 
solution has been found without proving their infeasibility either.
For each of them, the column \# reports the number of instances in that category.
The column CPU(s) reports the average running time (in seconds) to certify the optimality or 
the infeasibility, NODES the average number of branching nodes explored, and GAP(\%) 
the average mip gap of the unsolved instances. 

The results are grouped by benchmark set and then by feature generation methods. 
We remark that each benchmark set has been partitioned by feature generation methods since, 
given a specific feature, all the related generation methods described in 
Figure~\ref{fig:genschemes} create a partition of the benchmark sets. For example, taking 
into account the mandatory nodes, the CM and SM generation methods induce a partition on 
the set of instances considered. As a consequence, the instances considered in rows CM 
and SM on the one hand, and CPI and DPI on the other are the same, but partitioned 
differently.

\begin{table}[!ht]
\begin{adjustbox}{width=1\textwidth,center}
\begin{tabular}{@{}lrrrrrrrrrrrrr@{}}
\toprule
                &              &              & \multicolumn{3}{c}{\textbf{OPT}}                & \multicolumn{3}{c}{\textbf{NO OPT}}            & \multicolumn{3}{c}{\textbf{INFS}}              & \multicolumn{2}{c}{\textbf{NO SOLS}} \vspace*{1pt} \\
                &              & \textbf{\#} & \textbf{\#}  & \textbf{CPU(s)}  & \textbf{NODES} & \textbf{\#} & \textbf{NODES} & \textbf{GAP(\%)} & \textbf{\#} & \textbf{CPU(s)} & \textbf{NODES} & \textbf{\#}     & \textbf{NODES}     \\ \cmidrule(r){1-3} \cmidrule(lr){4-6} \cmidrule(lr){7-9} \cmidrule(lr){10-12} \cmidrule(l){13-14}
\textbf{SMALL}  & \textbf{ALL} & \textbf{156} & \textbf{156} & \textbf{17.50}   & \textbf{1765}  & \textbf{0}  & \textbf{-}     & \textbf{-}       & \textbf{0}  & \textbf{-}      & \textbf{-}     & \textbf{0}      & \textbf{-}         \\ \cmidrule(lr){2-3} \cmidrule(lr){4-6} \cmidrule(lr){7-9} \cmidrule(lr){10-12} \cmidrule(l){13-14}
\textbf{}       & CM  & 78           & 78           & 29.58            & 2959           & 0           & -              & -                & 0           & -               & -              & 0               & -                  \\
\textbf{}       & SM  & 78           & 78           & 5.41             & 571            & 0           & -              & -                & 0           & -               & -              & 0               & -                  \\
\textbf{}       & CPI  & 78           & 78           & 15.71            & 1620           & 0           & -              & -                & 0           & -               & -              & 0               & -                  \\
\textbf{}       & DPI  & 78           & 78           & 19.28            & 1911           & 0           & -              & -                & 0           & -               & -              & 0               & -                  \\ \cmidrule(r){1-3} \cmidrule(lr){4-6} \cmidrule(lr){7-9} \cmidrule(lr){10-12} \cmidrule(l){13-14}
\textbf{MEDIUM} & \textbf{ALL} & \textbf{120} & \textbf{108} & \textbf{655.12}  & \textbf{18148} & \textbf{7} & \textbf{96676} & \textbf{4.00}    & \textbf{2}  & \textbf{20.07}  & \textbf{0}     & \textbf{3}      & \textbf{43334}     \\ \cmidrule(lr){2-3} \cmidrule(lr){4-6} \cmidrule(lr){7-9} \cmidrule(lr){10-12} \cmidrule(l){13-14}
\textbf{}       & CM  & 60           & 54           & 865.84           & 26350          & 6           & 106596         & 3.34             & 0           & -               & -              & 0               & -                  \\
\textbf{}       & SM  & 60           & 54           & 444.41           & 9946           & 1           & 61954          & 6.33             & 2           & 20.07           & 0              & 3               & 43334              \\
\textbf{}       & CPI  & 60           & 55           & 590.02           & 17155          & 3           & 91326          & 5.85             & 1           & 22.20            & 0              & 1               & 50396              \\
\textbf{}       & DPI  & 60           & 53           & 722.68           & 19178          & 4           & 99350          & 3.08             & 1           & 17.94           & 0              & 2               & 36272              \\ \cmidrule(r){1-3} \cmidrule(lr){4-6} \cmidrule(lr){7-9} \cmidrule(lr){10-12} \cmidrule(l){13-14}
\textbf{LARGE}  & \textbf{ALL} & \textbf{120} & \textbf{1}   & \textbf{1607.82} & \textbf{43073} & \textbf{49} & \textbf{32541} & \textbf{5.12}    & \textbf{17} & \textbf{6.67}   & \textbf{0}     & \textbf{53}     & \textbf{12599}     \\ \cmidrule(lr){2-3} \cmidrule(lr){4-6} \cmidrule(lr){7-9} \cmidrule(lr){10-12} \cmidrule(l){13-14}
\textbf{}       & CM  & 60           & 1            & 1607.82          & 43073          & 41          & 34101          & 4.92             & 0           & -               & -              & 18              & 10991              \\
\textbf{}       & SM  & 60           & 0            & -                & -              & 8           & 17467          & 7.13             & 17          & 6.67            & 0              & 35              & 13261              \\
\textbf{}       & CPI  & 60           & 0            & -                & -              & 24          & 41693          & 5.23             & 8           & 6.60            & 0              & 28              & 13031              \\
\textbf{}       & DPI  & 60           & 1            & 1607.82          & 43073          & 25          & 26279          & 5.05             & 9           & 6.73            & 0              & 25              & 12043              \\ \cmidrule(r){1-3} \cmidrule(lr){4-6} \cmidrule(lr){7-9} \cmidrule(lr){10-12} \cmidrule(l){13-14}
\textbf{ALL}    & \textbf{}    & \textbf{396} & \textbf{265} & \textbf{283.36}  & \textbf{8598}  & \textbf{56} & \textbf{46620} & \textbf{4.88}    & \textbf{19} & \textbf{8.15}   & \textbf{0}     & \textbf{56}     & \textbf{13828}     \\ \bottomrule
\end{tabular}
\end{adjustbox}
\caption{The \topP{} results of the CPA.}
\label{PTOP-ST-MIN-BC}
\end{table}

Table~\ref{PTOP-ST-MIN-RESULTS} reports a comparison between CPLEX and the CPA. 
The table presents the same structure of Table~\ref{PTOP-ST-MIN-BC} but with a slightly 
different meaning. The columns CPU and NODES report the average relative differences,
$\Delta$ the absolute difference, and GAP(\%) the average absolute differences.
%
We remark that a negative value means that the CPA is better than CPLEX for the CPU time, number of 
nodes and mip gap difference. On the contrary, a positive value means that the CPA is actually 
better than CPLEX for the number of optimal solutions found.

\begin{table}[!hb]
\centering
\begin{adjustbox}{width=1\textwidth,center}
\begin{tabular}{@{}lrrrrrrrrrrrrr@{}}
\toprule
 &  &  & \multicolumn{3}{c}{\textbf{OPT}} & \multicolumn{3}{c}{\textbf{NO OPT}} & \multicolumn{3}{c}{\textbf{INFS}} & \multicolumn{2}{c}{\textbf{NO SOLS}} \vspace*{1pt} \\
 &  & \textbf{\#} & \textbf{$\Delta$} & \textbf{CPU} & \textbf{NODES} & \textbf{$\Delta$} & \textbf{NODES} & \textbf{GAP(\%)} & \textbf{$\Delta$} & \textbf{CPU} & \textbf{NODES} & \textbf{$\Delta$} & \textbf{NODES} \\ \cmidrule(r){1-3} \cmidrule(lr){4-6} \cmidrule(lr){7-9} \cmidrule(lr){10-12} \cmidrule(l){13-14}
\textbf{SMALL} & \textbf{ALL} & \textbf{156} & \textbf{0} & \textbf{-38.06\%} & \textbf{-108.66\%} & \textbf{0} & \textbf{-} & \textbf{-} & \textbf{0} & \textbf{-} & \textbf{-} & \textbf{0} & \textbf{-} \\ \cmidrule(lr){2-3} \cmidrule(lr){4-6} \cmidrule(lr){7-9} \cmidrule(lr){10-12} \cmidrule(l){13-14}
\textbf{} & CM & 78 & 0 & -42.32\% & -120.80\% & 0 & - & - & 0 & - & - & 0 & - \\
\textbf{} & SM & 78 & 0 & -14.77\% & -45.74\% & 0 & - & - & 0 & - & - & 0 & - \\
\textbf{} & CPI & 78 & 0 & -43.10\% & -121.11\% & 0 & - & - & 0 & - & - & 0 & - \\
\textbf{} & DPI & 78 & 0 & -33.96\% & -98.10\% & 0 & - & - & 0 & - & - & 0 & - \\ \cmidrule(r){1-3} \cmidrule(lr){4-6} \cmidrule(lr){7-9} \cmidrule(lr){10-12} \cmidrule(l){13-14}
\textbf{MEDIUM} & \textbf{ALL} & \textbf{120} & \textbf{0} & \textbf{-21.93\%} & \textbf{-80.32\%} & \textbf{1} & \textbf{-82.79\%} & \textbf{-0.46} & \textbf{0} & \textbf{4.68\%} & \textbf{0.00\%} & \textbf{-1} & \textbf{-133.37\%} \\ \cmidrule(lr){2-3} \cmidrule(lr){4-6} \cmidrule(lr){7-9} \cmidrule(lr){10-12} \cmidrule(l){13-14}
\textbf{} & CM & 60 & 0 & -34.73\% & -98.60\% & 1 & -70.58\% & -0.52 & 0 & - & - & -1 & - \\
\textbf{} & SM & 60 & 0 & -3.00\% & -31.88\% & 0 & -156.30\% & -0.23 & 0 & 4.68\% & 0.00\% & 0 & -133.37\% \\
\textbf{} & CPI & 60 & 0 & -23.97\% & -80.67\% & 2 & -81.62\% & -0.47 & 0 & 4.95\% & 0.00\% & -2 & -117.41\% \\
\textbf{} & DPI & 60 & 0 & -20.20\% & -79.99\% & -1 & -83.32\% & -0.45 & 0 & 4.35\% & 0.00\% & 1 & -155.54\% \\ \cmidrule(r){1-3} \cmidrule(lr){4-6} \cmidrule(lr){7-9} \cmidrule(lr){10-12} \cmidrule(l){13-14}
\textbf{LARGE} & \textbf{ALL} & \textbf{120} & \textbf{0} & \textbf{-63.05\%} & \textbf{-86.78\%} & \textbf{8} & \textbf{-37.64\%} & \textbf{-0.99} & \textbf{1} & \textbf{5.09\%} & \textbf{-} & \textbf{-9} & \textbf{-46.33\%} \\ \cmidrule(lr){2-3} \cmidrule(lr){4-6} \cmidrule(lr){7-9} \cmidrule(lr){10-12} \cmidrule(l){13-14}
\textbf{} & CM & 60 & 0 & -63.05\% & -86.78\% & 6 & -37.56\% & -1.07 & 0 & - & - & -6 & -53.72\% \\
\textbf{} & SM & 60 & 0 & - & - & 2 & -39.06\% & -0.30 & 1 & 5.09\% & - & -3 & - \\
\textbf{} & CPI & 60 & 0 & - & - & 8 & -41.12\% & -1.08 & 0 & 6.63\% & - & -8 & - \\
\textbf{} & DPI & 60 & 0 & -63.05\% & -86.78\% & 0 & -33.86\% & -0.94 & 1 & 3.58\% & - & -1 & -41.91\% \\ \cmidrule(r){1-3} \cmidrule(lr){4-6} \cmidrule(lr){7-9} \cmidrule(lr){10-12} \cmidrule(l){13-14}
\textbf{ALL} & \textbf{} & \textbf{396} & \textbf{0} & \textbf{-23.40\%} & \textbf{-83.87\%} & \textbf{9} & \textbf{-58.19\%} & \textbf{-0.88} & \textbf{1} & \textbf{4.98\%} & \textbf{-} & \textbf{-10} & \textbf{-57.24\%} \\ \bottomrule
\end{tabular}
\end{adjustbox}
\caption{The comparison between the CPA with CPLEX for the \topP{}.}
\label{PTOP-ST-MIN-RESULTS}
\end{table}

As reported in Table~\ref{PTOP-ST-MIN-BC}, the instances with scattered mandatory nodes (SM) 
appear easier to solve to optimality compared to those with clustered mandatory nodes (CM). 
As a matter of fact, in addition to the number of solved instances, the running time and 
the number of explored nodes are smaller for the SM instances. 
Such differences are more evident for medium size instances.

This behaviour could be explained by the following facts.
On one hand, the distances between mandatory nodes are expected to be large 
when they are scattered. On the other, if the mandatory nodes are clustered, 
the distances between them will tend to be smaller. 
This could imply that a single route may visit more mandatory nodes in those instances with
clustered mandatory nodes. Thus, the number of feasible solutions could be greater than those
in instances with scattered mandatory nodes, which makes them harder to solve to optimality 
(but easier in terms of finding a feasible solution).

With regard to the physical incompatibilities, there are no huge differences between the 
clusters-based (CPI) and degree-based (DPI) instances.

The results of the comparison between the CPA and CPLEX reported in 
Table~\ref{PTOP-ST-MIN-RESULTS} show that the CPA outperforms CPLEX,
which nevertheless seems a little bit faster than our algorithm to certify the infeasibility
of an instance. 
We also notice that the time needed to certify the optimality and the quality of 
the gaps provided for the instances not solved, are systematically in favour of 
the CPA, sometimes by a large amount.
It is worth noting that the CPA is capable to find $9$ more solutions
and to certify $1$ more infeasibility than CPLEX (this can be seen directly from the NO OPT and INFS columns).

In both tables, the time needed to certify the optimality and the quality of the gaps provided for the instances not solved, are systematically in favour of the CPA, sometimes by a large amount.

\begin{table}[!ht]
\begin{adjustbox}{width=1\textwidth,center}
\begin{tabular}{@{}lrrrrrrrrrrrrr@{}}
\toprule
                &              &           & \multicolumn{3}{c}{\textbf{OPT}}              & \multicolumn{3}{c}{\textbf{NO OPT}}            & \multicolumn{3}{c}{\textbf{INFS}}              & \multicolumn{2}{c}{\textbf{NO SOLS}} \vspace*{1pt} \\
                &              & \textbf{\#} & \textbf{\#}  & \textbf{CPU(s)}  & \textbf{NODES} & \textbf{\#} & \textbf{NODES} & \textbf{GAP(\%)} & \textbf{\#} & \textbf{CPU(s)} & \textbf{NODES} & \textbf{\#}     & \textbf{NODES}     \\ \cmidrule(r){1-3} \cmidrule(lr){4-6} \cmidrule(lr){7-9} \cmidrule(lr){10-12} \cmidrule(l){13-14}
\textbf{SMALL}  & \textbf{ALL} & \textbf{312}       & \textbf{312}         & \textbf{59.61}           & \textbf{3254}           & \textbf{0}           & \textbf{-}              & \textbf{-}                & \textbf{0}           & \textbf{-}               & \textbf{-}              & \textbf{0}               & \textbf{-}                  \\ \cmidrule(lr){2-3} \cmidrule(lr){4-6} \cmidrule(lr){7-9} \cmidrule(lr){10-12} \cmidrule(l){13-14}
\textbf{}       & CM  & 156       & 156         & 99.95           & 5453           & 0           & -              & -                & 0           & -               & -              & 0               & -                  \\
\textbf{}       & SM  & 156       & 156         & 19.27           & 1055           & 0           & -              & -                & 0           & -               & -              & 0               & -                  \\
\textbf{}       & CPI  & 156       & 156         & 59.99           & 3154           & 0           & -              & -                & 0           & -               & -              & 0               & -                  \\
\textbf{}       & DPI  & 156       & 156         & 59.23           & 3354           & 0           & -              & -                & 0           & -               & -              & 0               & -                  \\
\textbf{}       & FLI  & 156       & 156         & 55.05           & 2845           & 0           & -              & -                & 0           & -               & -              & 0               & -                  \\
\textbf{}       & NLI  & 156       & 156         & 64.42           & 3686           & 0           & -              & -                & 0           & -               & -              & 0               & -                  \\ \cmidrule(r){1-3} \cmidrule(lr){4-6} \cmidrule(lr){7-9} \cmidrule(lr){10-12} \cmidrule(l){13-14}
\textbf{MEDIUM} & \textbf{ALL} & \textbf{240}       & \textbf{137}         & \textbf{807.26}          & \textbf{11377}          & \textbf{90}          & \textbf{52525}          & \textbf{5.50}             & \textbf{0}           & \textbf{-}               & \textbf{-}              & \textbf{13}              & \textbf{25647}              \\ \cmidrule(lr){2-3} \cmidrule(lr){4-6} \cmidrule(lr){7-9} \cmidrule(lr){10-12} \cmidrule(l){13-14}
\textbf{}       & CM  & 120       & 61          & 1127.40         & 15841          & 59          & 56243          & 5.25             & 0           & -               & -              & 0               & -                  \\
\textbf{}       & SM  & 120       & 76          & 581.98          & 8236           & 31          & 43312          & 6.11             & 0           & -               & -              & 13              & 25647              \\
\textbf{}       & CPI  & 120       & 62          & 887.98          & 13219          & 51          & 56532          & 5.32             & 0           & -               & -              & 7               & 24835              \\
\textbf{}       & DPI  & 120       & 75          & 733.26          & 9689           & 39          & 47628          & 5.71             & 0           & -               & -              & 6               & 26621              \\
\textbf{}       & FLI  & 120       & 84          & 978.92          & 13588          & 29          & 53488          & 3.90             & 0           & -               & -              & 7               & 24316              \\
\textbf{}       & NLI  & 120       & 53          & 226.86          & 3903           & 61          & 52087          & 6.22             & 0           & -               & -              & 6               & 26755              \\ \cmidrule(r){1-3} \cmidrule(lr){4-6} \cmidrule(lr){7-9} \cmidrule(lr){10-12} \cmidrule(l){13-14}
\textbf{LARGE}  & \textbf{ALL} & \textbf{240}       & \textbf{1}           & \textbf{-}               & \textbf{-}              & \textbf{51}          & \textbf{17964}          & \textbf{12.21}            & \textbf{44}          & \textbf{39.78}           & \textbf{19}             & \textbf{144}             & \textbf{5724}               \\ \cmidrule(lr){2-3} \cmidrule(lr){4-6} \cmidrule(lr){7-9} \cmidrule(lr){10-12} \cmidrule(l){13-14}
\textbf{}       & CM  & 120       & 1           & -               & -              & 49          & 17964          & 12.21            & 0           & -               & -              & 70              & 5508               \\
\textbf{}       & SM  & 120       & 0           & -               & -              & 2           & -              & -                & 44          & 39.78           & 19             & 74              & 5910               \\
\textbf{}       & CPI  & 120       & 0           & -               & -              & 26          & 21449          & 12.26            & 22          & 62.58           & 37             & 72              & 5889               \\
\textbf{}       & DPI  & 120       & 1           & -               & -              & 25          & 14478          & 12.16            & 22          & 15.84           & 0              & 72              & 5559               \\
\textbf{}       & FLI  & 120       & 1           & -               & -              & 44          & 16876          & 10.92            & 19          & 13.46           & 0              & 56              & 6194               \\
\textbf{}       & NLI  & 120       & 0           & -               & -              & 7           & 32100          & 29.01            & 25          & 60.38           & 34             & 88              & 5457               \\ \cmidrule(r){1-3} \cmidrule(lr){4-6} \cmidrule(lr){7-9} \cmidrule(lr){10-12} \cmidrule(l){13-14}
\textbf{ALL}    & \textbf{}    & \textbf{792}       & \textbf{450}         & \textbf{233.31}          & \textbf{5141}           & \textbf{141}         & \textbf{40627}          & \textbf{7.81}             & \textbf{44}          & \textbf{39.78}           & \textbf{19}             & \textbf{157}             & \textbf{7195}               \\ \bottomrule
\end{tabular}
\end{adjustbox}
\caption{The \topPL{} results of the CPA.}
\label{PLTOP-ST-MIN-BC}
\end{table}

Tables~\ref{PLTOP-ST-MIN-BC} and~\ref{PLTOP-ST-MIN-RESULTS} are respectively
the equivalent of Tables~\ref{PTOP-ST-MIN-BC} and \ref{PTOP-ST-MIN-RESULTS}
but for \topPL{.}

\begin{table}[!hb]
\centering
\begin{adjustbox}{width=1\textwidth,center}
\begin{tabular}{@{}lrrrrrrrrrrrrr@{}}
\toprule
 &  &  & \multicolumn{3}{c}{\textbf{OPT}} & \multicolumn{3}{c}{\textbf{NO OPT}} & \multicolumn{3}{c}{\textbf{INFS}} & \multicolumn{2}{c}{\textbf{NO SOLS}} \vspace*{1pt} \\
 &  & \textbf{\#} & \textbf{$\Delta$} & \textbf{CPU} & \textbf{NODES} & \textbf{$\Delta$} & \textbf{NODES} & \textbf{GAP(\%)} & \textbf{$\Delta$} & \textbf{CPU} & \textbf{NODES} & \textbf{$\Delta$} & \textbf{NODES} \\ \cmidrule(r){1-3} \cmidrule(lr){4-6} \cmidrule(lr){7-9} \cmidrule(lr){10-12} \cmidrule(l){13-14}
\textbf{SMALL} & \textbf{ALL} & \textbf{312} & \textbf{8} & \textbf{-331.15\%} & \textbf{-468.83\%} & \textbf{-8} & \textbf{-} & \textbf{-} & \textbf{0} & \textbf{-} & \textbf{-} & \textbf{0} & \textbf{-} \\ \cmidrule(lr){2-3} \cmidrule(lr){4-6} \cmidrule(lr){7-9} \cmidrule(lr){10-12} \cmidrule(l){13-14}
 & CM & 156 & 4 & -289.25\% & -410.68\% & -4 & - & - & 0 & - & - & 0 & - \\
 & SM & 156 & 4 & -548.48\% & -769.31\% & -4 & - & - & 0 & - & - & 0 & - \\
 & CPI & 156 & 4 & -296.99\% & -478.89\% & -4 & - & - & 0 & - & - & 0 & - \\
 & DPI & 156 & 4 & -365.74\% & -459.36\% & -4 & - & - & 0 & - & - & 0 & - \\
 & FLI & 156 & 0 & -17.95\% & -59.09\% & 0 & - & - & 0 & - & - & 0 & - \\
 & NLI & 156 & 8 & -613.26\% & -802.20\% & -8 & - & - & 0 & - & - & 0 & - \\ \cmidrule(r){1-3} \cmidrule(lr){4-6} \cmidrule(lr){7-9} \cmidrule(lr){10-12} \cmidrule(l){13-14}
\textbf{MEDIUM} & \textbf{ALL} & \textbf{240} & \textbf{44} & \textbf{-40.43\%} & \textbf{-91.34\%} & \textbf{-32} & \textbf{-58.63\%} & \textbf{-5.78} & \textbf{0} & \textbf{-} & \textbf{-} & \textbf{-12} & \textbf{-134.00\%} \\ \cmidrule(lr){2-3} \cmidrule(lr){4-6} \cmidrule(lr){7-9} \cmidrule(lr){10-12} \cmidrule(l){13-14} 
 & CM & 120 & 22 & -47.80\% & -107.96\% & -19 & -57.83\% & -6.48 & 0 & - & - & -3 & - \\
 & SM & 120 & 22 & -30.38\% & -68.83\% & -13 & -61.20\% & -4.06 & 0 & - & - & -9 & -134.00\% \\
 & CPI & 120 & 17 & -33.20\% & -80.19\% & -10 & -57.61\% & -5.53 & 0 & - & - & -7 & -160.70\% \\
 & DPI & 120 & 27 & -48.45\% & -105.28\% & -22 & -60.11\% & -6.09 & 0 & - & - & -5 & -104.10\% \\
 & FLI & 120 & 12 & -23.67\% & -60.82\% & -10 & -58.65\% & -0.46 & 0 & - & - & -2 & -129.58\% \\
 & NLI & 120 & 32 & -284.88\% & -450.47\% & -22 & -58.62\% & -8.20 & 0 & - & - & -10 & -137.34\% \\ \cmidrule(r){1-3} \cmidrule(lr){4-6} \cmidrule(lr){7-9} \cmidrule(lr){10-12} \cmidrule(l){13-14}
\textbf{LARGE} & \textbf{ALL} & \textbf{240} & \textbf{1} & \textbf{-} & \textbf{-} & \textbf{2} & \textbf{-19.36\%} & \textbf{-2.44} & \textbf{3} & \textbf{-83.71\%} & \textbf{-245.36\%} & \textbf{-6} & \textbf{-54.80\%} \\ \cmidrule(lr){2-3} \cmidrule(lr){4-6} \cmidrule(lr){7-9} \cmidrule(lr){10-12} \cmidrule(l){13-14} 
 & CM & 120 & 1 & - & - & 0 & -19.36\% & -2.44 & 0 & - & - & -1 & -70.45\% \\
 & SM & 120 & 0 & - & - & 2 & - & - & 3 & -83.71\% & -245.36\% & -5 & -42.18\% \\
 & CPI & 120 & 0 & - & - & 2 & -19.72\% & -3.12 & 1 & -101.77\% & -245.36\% & -3 & -58.82\% \\
 & DPI & 120 & 1 & - & - & 0 & -18.82\% & -1.76 & 2 & -8.81\% & - & -3 & -50.53\% \\
 & FLI & 120 & 1 & - & - & -2 & -17.97\% & -1.69 & 1 & -4.10\% & - & 0 & -37.65\% \\
 & NLI & 120 & 0 & - & - & 4 & -28.80\% & -12.23 & 2 & -97.60\% & -245.36\% & -6 & -65.86\% \\ \cmidrule(r){1-3} \cmidrule(lr){4-6} \cmidrule(lr){7-9} \cmidrule(lr){10-12} \cmidrule(l){13-14}
\textbf{ALL} &  & \textbf{792} & \textbf{53} & \textbf{-97.45\%} & \textbf{-274.76\%} & \textbf{-38} & \textbf{-52.65\%} & \textbf{-4.63} & \textbf{3} & \textbf{-83.71\%} & \textbf{-245.36\%} & \textbf{-18} & \textbf{-75.64\%} \\ \bottomrule
\end{tabular}
\end{adjustbox}
\caption{The comparison between the CPA with CPLEX for the \topPL{}.}
\label{PLTOP-ST-MIN-RESULTS}
\end{table}

The results reported in Table~\ref{PLTOP-ST-MIN-BC} confirm also for \topPL{} 
the remarks made on the mandatory nodes (CM and SM) and on the physical incompatibilities 
(CPI and DPI) for the \topP{.}
Considering the logical incompatibilities, the instances generated with farthest logical 
incompatibilities (FLI) seem easier than the instances with nearest logical incompatibilities 
(NLI).
Indeed, the CPA is faster to solve the FLI instances on the small benchmark set, while
it solves a larger number of FLI instances to optimality on the medium benchmark set.
Finally, on the larger benchmark set where almost no instance is solved to optimality,
it manages to find more feasible solutions on the FLI than the NLI.
Due to the large number of NO SOLS, one should take the last remark with a grain 
of salt.

The results reported in Table~\ref{PLTOP-ST-MIN-RESULTS} confirm
the fact that the CPA outperforms CPLEX for \topPL{}, even more so than for \topP{.}
As a matter of fact, the CPA finds $53$ additional optimal solutions, certifies 
$3$ additional infeasibilities, and finds $18$ additional solutions compared to CPLEX, 
all the while using approximately 
half of the CPU time and exploring almost a quarter of the nodes.
Finally, we notice that the time needed to certify the optimality and the quality of 
the gaps provided for the instances unsolved are systematically in favour of 
the CPA, sometimes by a large amount.



\subsection{Results for TOP instances}
\label{results-top}

In this section, we would like to validate the CPA on the TOP against the state-of-art 
algorithms for the problem, which are those described in \cite{Assunção2020} and 
\cite{Hanafi2020} (to the best of our knowledge). 
We denote these algorithms with AM and HMZ, respectively.
To this end, we tested the CPA to solve the classical TOP benchmark set
by running a version of the algorithm in which all the TOP-ST-MIN specific 
ingredients are disabled and no further customisation has been done.

\begin{table}[!ht]
\begin{tabular}{@{}lrrrrrrrrrrrrr@{}}
\toprule
    & \multicolumn{1}{l}{}   & \multicolumn{6}{c}{\textbf{AM vs CPA}}  & \multicolumn{6}{c}{\textbf{HMZ vs CPA}} \vspace*{1pt} \\ 
\textbf{Set} & \textbf{\#} & \multicolumn{2}{c}{\textbf{OPT}} & \multicolumn{2}{c}{\textbf{CPU(s)}} & \multicolumn{2}{c}{\textbf{GAP(\%)}} & \multicolumn{2}{c}{\textbf{OPT}} & \multicolumn{2}{c}{\textbf{CPU(s)}} & \multicolumn{2}{c}{\textbf{GAP(\%)}} \\
\cmidrule(r){1-2} \cmidrule(lr){3-8} \cmidrule(l){9-14}
1  & 54  & 54          & 54          & 3.11     & \textbf{1.23}     & -    & -             & 54          & 54          & \textbf{0.98}    & 1.23          & -             & -   \\
2  & 33  & 33          & 33          & 0.22     & \textbf{0.07}     & -    & -             & 33          & 33          & 0.13             & \textbf{0.07} & -             & -   \\
3  & 60  & 60          & 60          & 165.59   & \textbf{63.04}    & -    & -             & 60          & 60          & \textbf{37.83}   & 63.04         & -             & -   \\
4  & 60  & 44          & \textbf{50} & 2479.98  & 2731.99          & 3.19  & 1.97          & \textbf{48} & 45          & 406.89           & 1578.95       & 1.48          & 2.45 \\
5  & 78  & 62          & 62          & 715.75   & \textbf{630.19}  & 3.01  & \textbf{2.74} & 62          & 62          & \textbf{324.20}  & 630.19        & \textbf{2.62} & 2.93 \\
6  & 42  & \textbf{42} & 41          & 475.26   & 489.39           & -     & 0.88          & 39          & \textbf{41} & 100.11           & 489.39        & 0.81          & 1.76 \\
7  & 60  & 47          & \textbf{49} & 1133.23  & 1252.87          & 1.94  & 1.59          & \textbf{50} & 48          & 464.98           & 1106.88       & 1.47          & 1.93 \\ \cmidrule(r){1-2} \cmidrule(lr){3-4} \cmidrule(lr){5-6} \cmidrule(lr){7-8} \cmidrule(lr){9-10} \cmidrule(lr){11-12} \cmidrule(l){13-14}
\textbf{ALL}           & 387         & 342      & \textbf{349}     & 692.49 & 785.87       & 2.76        & 1.79        & \textbf{346}     & 343           & 190.73        & 545.35  & 1.60    & 2.27 \\ \bottomrule
\end{tabular}
\caption{The comparison between our CPA and the competitors on the TOP instances. Bold entries highlight the best algorithm with respect to the number of solved instances. When the number of solved instances is the same for all algorithms, bold entries highlight the best results in terms of running time and mip gap.}
\label{tab:TOPresults}
\end{table}

The results reported in \cite{Assunção2020} are obtained running the
algorithm on a 4.0 GHz Intel Core i7-4790k while a 3.50 GHz Intel i7 
5930k processor is used in \cite{Hanafi2020}. Both algorithms were
run using $6$ threads and a time limit of 7200 seconds.
To provide a bit fairer comparison (we recall that our cpu frequency is 2.1 
GHz), we ran the CPA twice using a different time limit rescaled with respect 
to each competitor's cpu frequency.

Table~\ref{tab:TOPresults} reports the results of such a comparison. 
The column OPT reports the number of optimally solved instances, CPU(s) the average running 
time to certify the optimality, and GAP(\%) the average mip gap of the unsolved instances.
Each of these columns report on the right the results for the CPA while on the left the results for the algorithm competitor.


As can be seen from the Table~\ref{tab:TOPresults}, the CPA is competitive with respect to the two 
state-of-art algorithms. 
With respect to the AM, our algorithm computes more optimal solutions
on the sets $4$ and $7$, and less on the set $6$.
With respect to the HMZ, our algorithm computes more optimal solutions
on the set $6$ and less on the sets $4$ and $7$.
Overall, the CPA computes $7$ more optimal solutions than AM and $3$
less than HMZ.
The CPA average running time is the same for both tests except for 
Set 4 due to the $5$ additional optimal solutions while the average mip
gap decreases.
Overall, the CPA seems a little bit slower than AM and clearly slower 
than HMZ.


Finally, we mention that partial results, i.e., only for the Set 4, are provided 
for the Branch-Cut-and-Price approach of \cite{Pessoa2020}, which solves 55 instances 
out of 60 in one hour of running time on an Intel Xeon E5-2680 v3 running at 2.50 GHz.

\section{Conclusions}
\label{sec:end}


In this paper we addressed the Team Orienteering Problem with Service Times and Mandatory 
\& Incompatible Nodes (TOP-ST-MIN), a variant of the classic Team Orienteering Problem (TOP).
The TOP-ST-MIN includes three novel features that stem from two real-world problems 
previously studied by the authors. We considered two versions of the problem, that is the 
\topP{} and the \topPL{}.
We demonstrate the NP-completeness of the feasibility problem, which points to the
fact that the TOP-ST-MIN is harder to tackle than the classical TOP.
We reported a mixed formulation and a compact formulation for the two versions showing also 
that the compact one seems to be more efficient from a computational point of view.
Based on the compact formulation, we developed a Cutting-Plane Algorithm (CPA) exploiting 
five families of valid inequalities. Further, to the best of our knowledge, our separation 
scheme for the subtour elimination constraints results to be asymptotically faster than the 
classical methodology based on the resolution of multiple max-flow/min-cut problems.

We generated \emph{ad hoc} instances in such a way to evaluate the impact of the three 
features that characterise the problem.
Extensive computational experiments showed that the CPA outperforms CPLEX in solving the 
newly generated instances. In particular, the algorithm results to be particularly 
effective for the \topPL{} where it is able to solve 53 more instances than CPLEX. 
From our analysis, the presence of mandatory nodes and logical incompatibilities seem to 
be the features that complicate the problem the most.
The CPA is also competitive with respect to the state-of-art algorithms for the TOP since 
it is able to solve almost the same number of instances.

\printbibliography

\end{document}